\documentclass{amsart}

\usepackage{graphicx}
\usepackage{color}

\newcommand{\norm}[2]{\|{#1}\|_{{#2}}}
\newcommand{\tnorm}[1]{\|\kern-.4mm| {#1} |\kern-.4mm\|}
\newcommand{\ud}{\,{\rm d}}
\newcommand{\jump}[1]{[\kern-0.6mm [{#1}]\kern-0.6mm]}
\newcommand{\tjump}[1]{\lfloor\kern-0.4mm{#1}\kern-0.4mm\rfloor}
\newcommand{\av}[1]{\{\kern-1.2mm \{{#1}\}\kern-1.2mm\}}

\newcommand{\dee}{{\rm d}}

\newcommand{\su}{\sum_{T \in \Th}}

\newcommand{\Ainner}[2]{(\!({#1},{#2})\!)_A}
\newcommand{\Ahinner}[2]{(\!({#1},{#2})\!)_{A,h}}
\newcommand{\inner}[3]{({#1},{#2})_{{#3}}}
\newcommand{\Th}{\mathcal{T}_h}
\newcommand{\gint}{\Gamma_{\rm int}}
\newcommand{\gE}{\Gamma_0}
\newcommand{\gif}{\Gamma_-}
\newcommand{\Div}{{\rm div}}
\newcommand{\ndg}[1]{|\kern-.03cm \|{#1}|\kern-.03cm \|}

\newcommand{\cinvT}{C_{\partial}}
\newcommand{\cinv}{C_{\nabla}}

\newcommand{\chp}{C_{hp}}

\newcommand{\R}{\mathbb{R}}

\newcommand{\normal}{{n}}
\newcommand{\jumpu}[1]{ \lfloor #1 \rfloor }
\newcommand{\avg }[1]{ \av{ #1}}
\newcommand{\tf}{{t_{\rm f}}}
\newcommand{\dt}{\,{\rm dt}}

\newcommand{\man}[1]{{\color{black}{#1}}}

\newcommand{\ns}[1]{| \kern -.25mm \|{#1}| \kern -.25mm \|_{\rm S}}

\newcommand{\xn}{x_{n_2}^{T}}
\newcommand{\utr}{W}

\newcommand{\barhmin}{\bar{h}_{\min}}

\newtheorem{theorem}{Theorem}[section]
\newtheorem{lemma}[theorem]{Lemma}

\newtheorem{proposition}[theorem]{Proposition}
\theoremstyle{definition}

\theoremstyle{remark}
\newtheorem{remark}[theorem]{Remark}

\numberwithin{equation}{section}

\begin{document}

\title[Asymptotic numerical hypocoercivity of space-time dG]{Asymptotic numerical hypocoercivity \\ of the space-time discontinuous Galerkin method  for  {the} Kolmogorov equation}


%    author one information
\author{Zhaonan Dong}
\address{(Z.~Dong)  Inria, 48 Rue Barrault, 75589, Paris, France and CERMICS, ENPC, Institut Polytechnique de Paris, CNRS,  6 \& 8 avenue B.~Pascal, 77455 Marne-la-Vall\'{e}e, France.}
\email{zhaonan.dong@inria.fr}
\thanks{}

%    author two information
\author{Emmanuil H. Georgoulis}
\address{(E.~H.~Georgoulis) The Maxwell Institute for Mathematical Sciences \& Department of Mathematics, Heriot-Watt University,   Edinburgh EH14 4AS, United Kingdom, {\sc and} Department of Mathematics, School of Applied Mathematical and Physical Sciences, National Technical University of Athens, Zografou 15780, Greece, {\sc and} IACM-FORTH, Greece.}
\email{E.Georgoulis@hw.ac.uk}
\thanks{E.~H.~Georgoulis and P.~J.~Herbert gratefully acknowledge the financial support of EPSRC (grant number EP/W005840/2)}

%    author third information
\author{Philip J. Herbert}
\address{(P.~J.~Herbert) Department of Mathematics, University of Sussex, Brighton BN1 9RF, United Kingdom.}
\email{P.Herbert@sussex.ac.uk}

%    \subjclass is required.

\date{}

\dedicatory{}

%    Abstract is required.
\begin{abstract}
We are concerned with discretisations of the classical Kolmogorov equation by a standard space-time discontinuous Galerkin method.  The Kolmogorov equation serves as simple, yet rich enough in the present context, model problem for a wide range of kinetic-type equations: although it involves diffusion in one of the two spatial dimensions only, the combined nature of the first order transport/drift term and the degenerate diffusion are sufficient to `propagate dissipation' across the spatial domain in its entirety. This is a manifestation of the celebrated concept of hypocoercivity, a term coined and studied extensively by Villani in \cite{villani}.  We show that the  {classical} space-time discontinuous Galerkin method  {admits} a corresponding hypocoercivity property at the discrete level, asymptotically for large times.  To the best of our knowledge, this is the first result of this kind for any standard Galerkin scheme. This property is shown by proving one part of a discrete inf-sup-type stability result for the method in a family of norms dictated by a modified scalar product motivated by the theory in \cite{villani}. This family of norms contains the full gradient of the numerical solution, thereby allowing for a full spectral gap/Poincar\'e-type inequality at the discrete level, thus, showcasing a subtle, discretisation-parameter-dependent, numerical hypocoercivity property. Further, we show that the space-time discontinuous Galerkin method is inf-sup stable in the family of norms containing the full gradient of the numerical solution, which may be a result of independent interest.
\end{abstract}
%\keywords{Kolmogorov equation; hypocoercivity; numerical hypocoercivity, finite element methods;  discontinuous Galerkin; space-time methods.}

\maketitle


\section{Introduction}
In kinetic modelling, numerous degenerate evolution (integro-)differential equations model diffusion/dissipation along specific spatial directions only, without spanning the entirety of the spatial domain of definition of their solutions. Nonetheless, such models have been shown to exhibit convergence towards an equilibrium state over long time scales. Such models frequently emerge from integrating stochastic processes, often resulting into Kolmogorov/Fokker–Planck type partial differential equations (PDEs). A representative case includes the modelling of multi-particle systems whereby random collisions are permitted in the presence of external drift.

The term \emph{hypocoercivity}, introduced by Villani in \cite{villani},   {describes} the property of evolution (integro-differential) operators admitting dissipation also in spatial directions where no diffusion is explicitly present in the operator. Villani developed a framework for sufficient conditions for hypocoercivity in the celebrated work \cite{villani}, building upon earlier seminal works by H\'erau and Nier \cite{herau_nier}, Eckmann and Hairer \cite{jeckmann_hairer},  {H\'erau  \cite{herau2006}},  {Dolbeault, Mouhot and Schmeiser \cite{Dolbeault09}} Mouhot and Neumann \cite{mouhot_neumann} and others.  {We also mention the work by Talay \cite{Talay2002} as a key step in the development of these ideas.}
 A different framework for sufficient conditions for hypocoercivity was given by Dolbeault, Mouhot and Schmeiser in the seminal work \cite{DMS_TAMS}.

To fix ideas, in a Hilbertian setting with inner product $\langle \cdot, \cdot\rangle $ inducing a norm $\norm{\cdot}{}:=\sqrt{\langle \cdot, \cdot\rangle }$, we consider an evolution equation of the form
 \begin{equation}\label{PDE_intro}
 	u_t + \mathrm{L} u =0, \qquad\text{ whereby }\ \mathrm{L}:=\mathrm{L}_{skew} - \mathrm{L}_{diss},
 \end{equation}
 with $\mathrm{L}_{skew} $ corresponding to a skew-symmetric, typically ``transport-like'', operator and $\mathrm{L}_{diss}$  a degenerate, self-adjoint operator modelling diffusion/dissipation.  Due to the skew-symmetry of $\mathrm{L}_{ {skew}}$, a straightforward energy argument would imply
 \[
 \frac{1}{2}\frac{\ud}{\ud t}	\|u \|^2 =\langle  \mathrm{L}_{diss}u,u\rangle \le 0.
 \]
 The degenerate nature, however, of $\mathrm{L}_{diss}$ may \emph{not} be sufficient to provide a norm so that spectral gap/Poincar\'e-type inequality of the form
 \[
 \kappa \|u-\bar{u}\|^2\le -\langle  \mathrm{L}_{diss}u,u\rangle
 \]
 for some $\kappa>0$ holds, with $\bar{u}$ denoting an equilibrium state, or some kernel-eliminating functional. Thus, we \emph{cannot} immediately conclude
 \[
 \|u-\bar{u}\|\le{\rm e}^{-\kappa t_f^{}}\|u_0\|
 \]
 i.e., exponential decay to equilibrium, with $u_0$ denoting an initial state and $t_f$ a final time.

 However, such decay may, indeed, hold in certain settings, despite the presence of a non-trivial kernel for $ \langle  \mathrm{L}_{diss}u,u\rangle $, due to the degenerate nature of the dissipation. Such a state of affairs can occur, informally speaking, when $\mathrm{L}_{skew}$ is ``appropriately non-constant'' and is able to ``transport dissipation'' also in the directions where no diffusion is explicitly present, that is, the equation admits the hypocoercivity property.

 The proof of trend to equilibrium for the kinetic Fokker-Planck equation was given by H\'erau \& Nier in \cite{herau_nier} upon modifying the inner product used to include certain mixed derivatives. Related ideas have also been used by Eckmann \& Hairer in \cite{jeckmann_hairer} for the study of spectral properties of certain hypoelliptic operators involving degenerate diffusions, and by Mouhot \& Neumann \cite{mouhot_neumann} in the study of kinetic models with integral-type collision operators, among others.  Villani \cite{villani} and Dolbeault, Mouhot, \& Schmeiser \cite{DMS_TAMS} conceptualised the algebraic nature of the interplay between transport and dissipation properties of the spatial operator $\mathrm{L}$ and were able to provide abstract sufficient conditions for hypocoercivity to hold, that is,	$
 \|u-\bar{u}\|\le C{\rm e}^{-\kappa t_f^{}}\|u_0\|
 $,
 for some constant $C>0$. They achieved this by modifying the energy setting, i.e., by constructing appropriate inner products of the form $\langle \cdot,\cdot \rangle_{hc}:=\langle \cdot,\cdot \rangle + (\!( \cdot,\cdot )\!)$,  {with $(\!( \cdot,\cdot )\!)$ denoting an additional appropriately constructed symmetric bilinear form,} whose respective induced norms are strong enough to admit spectral gap/Poincar\'e-type inequality and, hence, decay can be inferred. The modified inner products $\langle \cdot,\cdot \rangle_{hc}$ are carefully constructed to reflect the interplay between  $\mathrm{L}_{skew}$ and $\mathrm{L}_{diss}$.

 {It is highly desirable to preserve hypocoercive structures upon discretisation} to guarantee that the resulting numerical schemes   {demonstrate} robust performance for long-time computations. However, the exploration of \emph{numerical hypocoercivity} properties remains relatively limited to date. Notable works in this area include \cite{porretta_zuazua}, which investigates hypocoercivity for a central difference discretisation of the Kolmogorov equation (see \eqref{eq:kolmogorov} below) within the abstract framework of \cite[Theorem 18]{villani}. Further, Dujardin, H\'erau, and Lafitte \cite{dujardin_herau_lafitte} extend this approach to examine the spectral gap for central difference discretisations of the Fokker-Planck equation. Foster, Loh\'eac, and Tran \cite{foster_loheac_tran} present a Lagrangian-type splitting method employing  a similarity transformation and linear finite elements on quasi-uniform meshes.  Bessemoulin-Chatard and Filbet \cite{besse_filbet} propose design principles for constructing equilibrium-preserving finite volume methods for nonlinear degenerate parabolic problems. Recently, Bessemoulin-Chatard, Herda, and Rey \cite{besse_FV} introduce an asymptotic-preserving finite volume scheme for one-dimensional kinetic equations in the diffusive limit, applying the hypocoercivity framework of Dolbeault, Mouhot, and Schmeiser \cite{DMS_TAMS}. For a finite-volume method of a one-dimensional nonlinear reaction kinetic model, we refer \cite{bessemoulin2024discrete}.

With the exception of \cite{foster_loheac_tran} whereby kinetic equations in two spatial dimensions are discussed, the above developments on numerical methods concern one-dimensional spatial problems and/or finite difference or finite volume methods. Developing Galerkin/finite element methods that preserve hypocoercivity structures remains a largely unexplored area, despite the advantages these methods offer, such as high-order discretisations over unstructured or locally adapted meshes.

The first work in this context is \cite{georgoulis}, whereby a hypocoercivity-preserving family of spatially discrete non-conforming finite element methods of arbitrary order for the Kolmogorov equation
\begin{equation}\label{eq:kolmogorov}
 u_t+ L u = 0,\quad \text{in }(0,t_f]\times \man{\Omega};
\end{equation}
 {with $L u :=  - u_{xx}+ x u_y$}, for some final time $t_f>0$, with $\Omega\subset\mathbb{R}^2$, along with appropriate initial and boundary data, is constructed and analysed, following the framework in \cite{villani}. Recently, a hypocoercivity-exploiting fully-discrete, stabilised space-time finite element method  is presented and analysed in \cite{DonGeoHer24}, which is proven to admit ``numerical hypocoercivity'' in an asymptotically exact sense. In both \cite{georgoulis,  DonGeoHer24}, Kolmogorov equation serves as a simple model problem characterised by diffusion in one of the two spatial directions only, and, moreover, it admits a hypocoercivity-inducing algebraic structure in the spirit of  \cite{villani}.

 Due to their relevance in setting the context of the present work, we give a conceptual overview of the developments in \cite{georgoulis,  DonGeoHer24}.
	 {Upon defining a} symmetric and positive semidefinite matrix $A\in \mathbb{R}^{2\times 2}$,  {(see \eqref{eq:A_def} below for details,)} we denote by $\Ainner{\cdot}{\cdot}$ the inner product
$
	\Ainner{w}{v}:=(w,v)+ (\nabla w,A\nabla v ),
	$
	for $w,v\in H^1(\Omega)$, with induced norm $\norm{\cdot}{A}:= \sqrt{\Ainner{\cdot}{\cdot}}$; here $(\cdot,\cdot)_{\omega}$ corresponds to the standard $L_2(\omega)$ inner product for $\omega\subset \mathbb{R}^d$, $d\in\mathbb{N}$, along with respective induced norm $\|\cdot \|_\omega$ and we drop the subscript when $\omega=\Omega$, for brevity. Hence, assuming sufficient regularity, \eqref{eq:kolmogorov} implies
		\begin{equation}\label{eq:weak_hypo}
		\Ainner{u_t}{v}+ \Ainner{L u}{v}		=0,
	\end{equation} for all $v$ in  {a} suitable test space.  {Roughly speaking, i.e., ignoring non-trivial boundary terms and regularity considerations, \eqref{eq:weak_hypo} can be viewed as arising from testing \eqref{eq:kolmogorov} against $v-\nabla\cdot A\nabla v$.}
Imposing, now, homogeneous boundary conditions on the appropriate part of $\partial\Omega$,  the latter gives
	\begin{equation}\label{eq:weak_bilinear}
		\Ainner{u_t}{v} + a(u,v) =0, \quad\text{with }\ \ 	a(u,v):=(u_x,v_x)+(xu_y,v)+(\nabla L u, A\nabla v).
	\end{equation}
	Integration by parts  and elementary operations eventually give (see Section \ref{sec:weak_form} below for details):
		\begin{equation}\label{eq:hypoco3old}
			a(u,u)
		 \ge   c \Big(  \|u_x\|^2
			+
			 \|u_y\|^2
			+ \norm{\sqrt{A}\nabla u_x}{}^2			+(\frac{1}{2}xn_2\nabla u-n_1\nabla u_{x}, A\nabla u)_{\partial\Omega}\Big).
	\end{equation}

	If the  {entries} of $A$ are  taken to be fixed suitable numbers, as done in  \cite{georgoulis}, and additional boundary conditions are to treat the last term on the right-hand side of the above lower bound, we deduce
	$
		a(u,u)\ge   c_0 \norm{u}{A}^2
	$,
for some $c_0>0$, independent of $u$. The latter, in turn, implies
		\[
\frac{1}{2}	\frac{\ud }{\ud t } \norm{u}{A}^2+  c_0 \norm{u}{A}^2\le 0,
	\]
	i.e., exponential decay as $t_f\to\infty$. (We note that, in the present setting of a bounded spatial domain, decay results from the additional boundary conditions.) Further, in  \cite{georgoulis} a  {$C^0$-interior penalty} spatially-discrete (non-conforming) finite element method with continuous elements is designed and corresponding decay is proven, thereby yielding the first, to the best of our knowledge, Galerkin method admitting a ``numerical hypocoercivity'' property. The respective spatial bilinear form $a(\cdot,\cdot)$ corresponds to a fourth order operator; this remains the case upon Galerkin discretisation  {if the entries of the matrix $A$ are selected independent of the discretisation parameters}; we refer to \cite{georgoulis} for details.

	Starting again from \eqref{eq:weak_bilinear}, a stabilised non-conforming finite element method for the fully-discrete Kolmogorov problem (without additional boundary conditions) was recently developed in \cite{DonGeoHer24}, combining the classical streamline upwind Petrov-Galerkin (SUPG) point of view with the above modified hypocoercivity-inducing inner product to arrive at a non-standard spatial discretisation. The method is completed by a discontinuous Galerkin time-stepping scheme, to arrive at a fully-discrete space-time Galerkin method of arbitrary order. To ensure that the spatially discrete operator is spectrally equivalent to a second order discrete operator, the entries of the matrix $A$ in \cite{DonGeoHer24} are scaled by the local mesh-size and polynomial degrees of the Galerkin spaces.  For this method, the asymptotically exact numerical hypocoercivity with mesh and polynomial degree dependent exponential decay rate is proven  {and} a complete a priori error analysis is given \cite{DonGeoHer24}.

In this work, we demonstrate that the  {classical} space–time discontinuous Galerkin (dG) method— {namely, the symmetric interior penalty dG method in the $x$-variable combined with an upwind dG treatment in the $y$- and $t$-variables}, applied to \eqref{eq:kolmogorov} preserves an asymptotic numerical hypocoercivity structure.  In stark contrast to \cite{georgoulis,  DonGeoHer24}, the asymptotic numerical hypocoercivity is shown through the establishment of one side of an inf-sup condition on a sufficiently strong norm. This special norm is related to the SUPG-hypocoercive norm from \cite{DonGeoHer24}. We stress that the result holds for  the classical space-time dG scheme \cite{hss}, i.e., \emph{without} the presence of additional $A$-inner-product-induced Petrov-Galerkin-type terms, as is the case is \cite{georgoulis,  DonGeoHer24};  cf., \eqref{eq:weak_hypo}. To the best of our knowledge, this is the first demonstration of (asymptotic) numerical hypocoercivity for a well-known/classical Galerkin method.

 The classical, fully-discrete discontinuous Galerkin (dG) method for advection or advection-diffusion PDEs is known to be inf-sup stable in an  SUPG-like norm \cite{johnson_pitkaranta, ErnGuermoond06, ayuso_marini, MR2257111, cdgh_book, MR3672375} involving control in the advection direction. A key idea of the present developments is the observation that a term of the form $V+\tau (V_t+xV_y)-\nabla\cdot A\nabla V$ lies in the same discontinuous polynomial space as $V$ itself, for $\tau$ and $A$ element-wise constant functions. This observation eventually allows for coercivity in a norm strictly stronger than the SUPG-like norms for which dG is known to be inf-sup stable. For the fully-discrete space-time dG method, we also complete the inf-sup stability in the new `enhanced' norm.  The new norm is strong enough to counteract the absence of  spectral gap/Poincar\'e inequality of the standard $L_2$-energy argument and, thereby, to conclude asymptotic numerical hypocoercivity for the classical space-time dG scheme. To the best of our knowledge, the latter constitutes a stronger stability result to what is available in the literature to date for the space-time dG method for the Kolmogorov initial/boundary value problem. The method of proof is generic and is expected to be applicable to other related equations also.

The remainder of this work is structured as follows. In Section \ref{sec:weak_form}, upon providing some necessary notation, we discuss in more detail the hypocoercivity mechanism for the model problem of the Kolmogorov equation. In Section \ref{sec:method}, we present its discretisation by the space(-time) discontinuous Galerkin method. In Section \ref{sec:semi-discrete}, we present one part of the crucial inf-sup condition which, in turn, allows to conclude asymptotic numerical hypocoercivity of the spatially discrete method. In Section~\ref{sec:fullDiscrete}, we extend the analysis to the fully-discrete space--time dG method. In particular, we establish the full inf--sup stability result and derive corresponding \emph{a priori} error bounds for sufficiently regular solutions. In Section~\ref{sec: numerical experiments}, we present a series of numerical experiments confirming the theoretical developments.
Finally, in Section \ref{sec: conclusion and extension}, we draw some conclusions and give some potential extensions.

\subsection*{Notation}
We abbreviate the $L_2(\omega)$-inner product and $L_2(\omega)$-norm for a Lebesgue-measurable subset $\omega\subset \mathbb{R}^d$ as $(\cdot,\cdot)_{\omega}$ and $\norm{\cdot}{\omega}$, respectively.
Moreover, when $\omega=\Omega$, we will further compress the notation to $(\cdot,\cdot)\equiv (\cdot,\cdot)_{\Omega}$ and $\norm{\cdot}{}\equiv \norm{\cdot}{\Omega}$.
The standard notation $H^k(\omega)$ for Hilbertian Sobolev spaces, $k\in\mathbb{R}$ will be used.
In addition, given an interval $J\subset \mathbb{R}$ and a Banach space $V$, we use standard notation  {for} Bochner spaces, denoted by $L_p(J; V)$.

\section{Hypocoercive structure} \label{sec:weak_form}

Closing a degenerate parabolic PDE with suitable boundary conditions is well understood via the classical theory of linear second order equations with non-negative characteristic form \cite{fichera,or73}.
Assuming that $\partial \Omega$ is polygonal and denoting the unit outward normal at each point of $\partial \Omega$ (whenever well defined) by $\mathbf{n}:=(n_1,n_2)^T$ at almost every point of $\partial\Omega$, we first define the elliptic portion of the boundary
\[
	\gE:=\{(x,y)\in \partial \Omega: n_1(x,y)\ne 0\}.
\]
We further subdivide $\partial \Omega\backslash\gE$, into inflow and outflow boundaries
\[
	\gif:= \{(x,y) \in \partial\Omega\backslash\gE: xn_2(x,y )<0\}, \quad \Gamma_+:= \{\xi \in  \partial\Omega\backslash\gE: xn_2(x,y)   \ge  0\},
\]
respectively, so that
 $\partial \Omega = \gE\cup\gif\cup\Gamma_+$.
Introducing the time interval $I : = (0, t_f]$, for a final time $t_f>0$,  and the notation $ L:= -\partial^2_{x}+x\partial_y$, the initial/boundary-value Kolmogorov problem we consider is given by:
	\begin{equation}\label{eq:kolmogorov-bounded}
	\begin{aligned}
u_t+Lu =&\ f, &\text{in } I\times\Omega,\\
u=&\ u_0, &\text{on } \{0\}\times\Omega,\\
u=&\ 0, &\text{on } I\times\gif,\\
n_1 u_x=&\ 0, &\text{on } I\times \gE,
\end{aligned}
\end{equation}
for $f\in L_2(I; L_2(\Omega))$, and $u_0\in L_2(\Omega)$.
The well-posedness of the above problem is assured upon assuming that $\gif$ has positive one-dimensional Hausdorff measure \cite{or73}. We note  that different boundary conditions  {on} $I\times \gE$ are admissible and give  {results}  to what is proven below under minor modifications only; we do not consider these explicitly for brevity.

\subsection*{The matrix $A$}

Throughout this work $A$ will be the matrix
\begin{equation}\label{eq:A_def}
 A : =      \begin{pmatrix}
     \alpha & \beta \\ \beta & \gamma
    \end{pmatrix}
\end{equation}
where the values of $\alpha$, $\beta$, and $\gamma$ will be given as
\begin{equation}\label{eq:alpha_beta_gamma_def}
    \alpha =
    \frac{1}{8\delta},
    \quad
    \beta =
    \frac{1}{48 \delta^2},
    \quad
    \gamma =
    \frac{10}{48^2 \delta ^3}
\end{equation}
where $\delta > 0$ is a parameter  {depending on the discretization parameters}, which will be chosen later and will vary (slightly) between the semi- and fully-discrete problems.
With this choice of $\alpha$, $\beta$, and $\gamma$, we see that $A$ is positive definite with $\det(A) = (9216 \delta^4)^{-1}$.
The choice of parameters given in \eqref{eq:alpha_beta_gamma_def} is by no means unique, as long as the analysis presented below remains true.  We opt to provide a concrete choice of $A$, rather than conditions on the entries $\alpha$, $\beta$, and $\gamma$, in an effort to streamline the analysis below.
In \cite{georgoulis}, the values of $\alpha,\beta,\gamma$ are taken to be fixed numbers, whereas, in \cite{DonGeoHer24}, $\delta$ is selected in a mesh and polynomial degree dependent fashion, as we shall also do so below. With this matrix $A$, we denote by $\Ainner{\cdot}{\cdot}$ the inner product
\[
	\Ainner{w}{v}:=(w,v)+ (\nabla w,A\nabla v ),
\]
for $w,v\in H^1(\Omega)$ with induced inner product $\norm{\cdot}{A}:= \sqrt{\Ainner{\cdot}{\cdot}}$.

\subsection*{Hypocoercivity}
Consider a test function $v\in H^1_{-}(\Omega):=\{v\in H^1(\Omega):v|_{\partial_{-}\Omega}=0\}$. Assume that the exact solution, $u$, of \eqref{eq:kolmogorov-bounded} is sufficiently regular, so that the following calculations are valid.
Upon differentiation of \eqref{eq:kolmogorov-bounded} and testing against $A\nabla v$, and subsequent addition to \eqref{eq:kolmogorov-bounded} tested against $v$ and integration by parts, the problem can be written in the non-standard weak form: find $u\in H^1_{-}(\Omega)$, such that
\begin{equation}\label{modified_weak_form}
	\Ainner{u_t}{v} + a_{\rm hc}(u,v) = \Ainner{f}{v},
\end{equation}
whereby
\begin{equation}\label{eq:spatial_part}
	a_{\rm hc}(u,v):=(u_x,v_x)+(xu_y,v)+(\nabla L u, A\nabla v).
	\end{equation}
	 {Again, momentarily assuming that $u$ is regular enough, specifically $u\in H^1_-(\Omega)\cap H^{2.5+\epsilon}(\Omega) $, so that  $(\frac{1}{2}xn_2\nabla u-n_1\nabla u_{x}, A\nabla u)_{\partial\Omega}$ is well defined, we evaluate $ {a_{\rm hc}(u,u)}$ and perform integrations by parts, to deduce}
\begin{equation}\label{eq:hypoco3}
	\begin{aligned}
		a_{\rm hc}(u,u)=&\ \norm{u_x}{}^2+\frac{1}{2}\norm{\sqrt{xn_2}u}{\Gamma_+}^2+\norm{\sqrt{A}\nabla u_x}{}^2+\beta\norm{u_y}{}^2+\alpha(u_y, u_x)\\
		&+(\frac{1}{2}xn_2\nabla u-n_1\nabla u_{x}, A\nabla u)_{\partial\Omega};
	\end{aligned}
\end{equation}
 {the detailed calculation is given in \cite{georgoulis}. Moreover, in }\cite{georgoulis}, additional boundary conditions are considered, so that the unsigned boundary term on right-hand side of \eqref{eq:hypoco3} is non-negative and, thus, can be omitted.  {In this work, we do \emph{not} assume any such condition and} we will make $A$ small, so that the boundary terms may be absorbed in the positive terms  {at the discrete level}, as done also in \cite{DonGeoHer24}.

Returning to the continuous problem, for the sake of highlighting the `mechanics' of hypocoercivity, we assume \emph{momentarily} that $(\frac{1}{2}xn_2\nabla u-n_1\nabla u_{x}, A\nabla u)_{\partial\Omega}\ge 0$.  This can be achieved, for instance, if $\nabla u=\mathbf{0}$ on $\partial\Omega$; we refer to \cite{georgoulis} for other possibilities. Now, applying Young's and Cauchy-Schwarz' inequalities, and upon discarding  non-negative terms, we deduce
\begin{equation}\label{eq:hypoco}
    a_{\rm hc}(u,u)\ge
    \frac{1}{4}\|u_x\|^2
    +
    \frac{1}{64 \delta^2} \|u_y\|^2,
\end{equation}
by the choices of $\alpha$ and $\beta$ made in \eqref{eq:alpha_beta_gamma_def}.
It then follows that there is $c>0$, depending on $\delta$ such that $a_{\rm hc}(u,u) \geq c\norm{u}{A}^2$, thus one has
\begin{equation}
\frac{1}{2}	\frac{\ud }{\ud t } \norm{u}{A}^2+ c \norm{u}{A}^2 \leq \Ainner{f}{u},
\end{equation}
which allows for stability estimates that do not depend exponentially on the final time $t_f$. In other words, using the non-standard test function $v-\nabla\cdot (A\nabla v)$ for well-chosen $A$, results in $L$ providing coercivity also with respect to $\|u_y\|$.
This is a manifestation of the concept of hypocoercivity: the nature of the trasport operator $x\partial_y$ is such that it can ``disseminate'' the available dissipation  of the problem, $-\partial_{x}^2$, in all spatial directions. Villani showed in \cite{villani} that sufficient conditions for a degenerate operator admitting such hypocoercivity property can be of purely algebraic nature and, in particular, they concern the study of linear spaces spanned by certain commutators between the transport and the dissipation parts of the operator. In the particular case of \eqref{eq:kolmogorov}, since $[x\partial_y,\partial_x]=\partial_y$, we can expect hypocoercive structure since $\partial_x$ and $\partial_y$ correspond to directional derivatives along the axes and, thus, they span $\Omega$; we refer to \cite[Theorem 24]{villani} for details.

\section{Discrete spaces and problem formulation} \label{sec:method}
\subsection{Space discretisation}
For the rest of this work, $\Omega$ is a Lipschitz polygonal domain. We introduce a family of subdivisions $\mathcal{T}_H=\{\Th\}_{h\in H}$, for some index set $H$,  in $\Omega$, consisting of mutually disjoint open triangular or box-type elements $T\in \Th $, whose closures cover $\bar{\Omega}$ exactly.
Also, let $h\colon\cup_{T\in\Th } T\to\mathbb{R}_+$ be the local meshsize function defined elementwise by $h|_T:=h_T:={\rm diam}(T)$.
Each $\Th $ is assumed to be shape-regular, in the sense that the radius $\rho_T$ of the largest inscribed circle of each $T\in\Th $ is bounded from below with respect to each element's diameter $h_T$, uniformly in the index set $H$.
Also, we assume that each $\Th $ is locally quasi-uniform in the sense that the diameters of adjacent elements are uniformly bounded from above and below in the index set $H$. Finally, for each $\Th $, let $\Gamma:=\cup_{T\in\Th }\partial T$ denote the mesh skeleton, and define
$
    \gint : = \{x \in \Gamma : x \not \in \partial \Omega\},
$
the interior part of the mesh skeleton.

For each element  $T\in \Th $, we define its inflow and outflow boundaries,
\[
\partial_- T := \{ (x,y) \in \partial T :  x  n_2(x,y) <0 \} \mbox{, }\quad  \partial_+ T := \{ (x,y) \in \partial T :  x n_2(x,y)  \geq 0 \},
\]
respectively.
Given $T\in \Th $, the trace of a function $w \in H^1(T)$ on $\partial_- T$, relative to $T$, is denoted by $w_{T}^+$. Further, if $\partial_- T \backslash    \gif \neq \emptyset$, then for $x\in \partial_- T\backslash    \gif$, there exists a unique $T^\prime \in \Th $ such that $x\in \partial_+ T^{\prime}$; with this notation, we denote by $w_{T}^-$ the trace of $w|_{T^\prime}$ on $\partial_- T \backslash \gif$. Hence the \emph{upwind jump} of the (scalar-valued) function $w$ across a face $F\in \partial_{-}T \backslash \gif$ is denoted by
$$
\jumpu{w}: = w^+_T - w^-_T.
$$
We define the \emph{jump} $\jump{w}$ of $w$ across the common element interface $e:=\partial T^i\cap \partial T^j$ of two adjacent elements $T^i,T^j$, by $\jump{w}|_e:=w|_{T^i}-w|_{T^j}$, when $i>j$, for the indices; correspondingly, we also define the \emph{average} $\av{w}|_e:=(w|_{T^i}+w|_{T^j})/2$ of $w$ across $e$.
In addition, for boundary faces $e\subset \partial T\cap \partial\Omega$, we set $\jump{w}|_e:=w|_{T}$ and $\av{w}|_e:=w|_{T}$, respectively. The above jump and average definitions are trivially extended to vector-valued functions by component-wise application.

We  {denote the \emph{broken Sobolev spaces}, subordinate to a mesh $\Th$ by
$
H^r(\Omega,\Th ):=\{v|_T\in H^r(T), T\in\Th \}.
$}
%with respective norm  {$\norm{w}{H^r(\Omega,\Th )}$}; see, e.g., . %:=\Big(\sum_{T\in\Th }\norm{\nabla^r w}{T}^2+ \sum_{k=1}^r\norm{\av{h}^{r- k-1/2}\jump{\nabla^{k-1} w}}{\Gamma_{\rm int}}^2\Big)^{1/2},\quad  r\ge 1.\]
Further, we define the broken gradient $\nabla_h $ to be the element-wise gradient operator with $\nabla_hw|_T=\nabla w|_T$, $T\in\Th $, for $w\in H^1(\Omega,\Th )$. Similarly, we define the broken divergence operator $\Div_h$ to be the element-wise divergence operator.

We consider the family of discontinuous element-wise polynomial spaces subordinate to $\Th $, defined by
\[
 {\mathcal{V}}_h\equiv  {\mathcal{V}}_h^p:=\{V\in L_2(\Omega): V|_T\in \mathbb{P}_p(T),\ T\in\Th \},
\]
with $\mathbb{P}_p(\omega)$, $\omega\subset\mathbb{R}^d$, denoting the space of polynomials of total degree at most $p$ over $\omega$, $p\in\mathbb{N}$.

We will use the standard $hp$-version inverse inequalities (see, e.g., \cite{schwab}):
\begin{equation}\label{trace_H1_inv}
\|V\|_{\partial T} \leq \frac{\cinvT p}{ \sqrt{h_T}}  \|V\|_{ T},
 \quad\text{and}\quad
\|\nabla V\|_{T} \leq  \frac{\cinv p^2}{h_T}  \|V\|_{ T},
\end{equation}
which hold of for all $V\in\mathbb{P}_p(K)$, $p\in\mathbb{N}$, $T\in\Th $,
with the constants $\cinvT>0$ and $\cinv>0$, depending only on the space dimension $d$ and  the mesh shape-regularity constant $C_\rho>0$, defined as the global constant, such that $h_{T}/h_{T'}\le C_{\rho}$,
for all $T,T'\in\Th$ sharing a $(d-1)$-face, and for all $\{\Th\}_{h\in H}$.

For spatial discretisation, we select the \emph{classical} (upwind) discontinuous Galerkin method of Houston, Schwab and S\"uli \cite{hss} for equations with non-negative characteristic form. For the particular case of the Kolmogorov equation, the dG bilinear form from from \cite{hss} discretising $Lu$ reads:
\begin{equation}\label{def: spatial bilinear form}
\begin{split}
    a_{\rm dG}(U,V):=&
    \sum_{T \in \Th} \left( \inner{U_x}{V_x}{T} + \inner{x U_y}{V}{T}\right)\\
  &
  -
    \inner{\av{\normal_1 U_x}}{\jump{V}}{\gint }
    -
    \inner{\av{\normal_1 V_x}}{\jump{U}}{\gint }
    +
    \inner{\sigma \jump{U}}{\jump{V}}{\gint}
    \\
    &-
    \sum_{T \in \Th} \left(
    \inner{x\normal_2 \jumpu{U}}{V^+}{\partial_- T \cap \gint}
    +
    \inner{x\normal_2 U^+ }{ V^+ }{\partial_- T \cap \gif} \right),
\end{split}\end{equation}
where $\sigma\colon \gint \to \R$ is a penalty term whose precise expression is to be provided later. So, the spatially discrete method reads: find $U\in H^1(I; {\mathcal{V}}_h)$, such that
\begin{equation}\label{def: seim-discrete bilinear form}
    B_{\rm s}(U,V)
    :=
    \int_0^{\tf}  \left(\inner{U_t}{V}{} +  a_{\rm dG}(U,V) \right) \dt
    =
    \int_0^\tf
    \inner{f}{V}{}=:   \ell_{\rm s}(V)
    \end{equation}
for all $V\in  {\mathcal{V}}_h$, and $U(t_0): = \pi u_0$, with $\pi\colon L_2(\Omega) \rightarrow  {\mathcal{V}}_h$ denoting, the orthogonal $L_2$-projection onto $ {\mathcal{V}}_h$.

\subsection{Time discretisation}

We subdivide $I=(0,t_f]$, into the family $\mathcal{I} := \{ I_n : n = 1,\ldots, N\}$, where $I_n := (t_{n-1}, t_n]$ for a strictly increasing sequence $0 = t_0 < t_1 < \ldots < t_N = t_f$, and let $k\colon I\to\mathbb{R}_+$, with $k|_{I_n} := k_n = t_n - t_{n-1}$ denote each timestep.
Further, we define a family of space-time (discontinuous) finite element spaces subordinate to $\mathcal{I}$ and $\Th $ by
\begin{equation*}
	 {\mathcal{V}}_{h,k}\equiv	 {\mathcal{V}}_{h,k}^{p,q} := \{ V \in L_2(I;  {\mathcal{V}}_h^p) : V|_{I_n} \in \mathbb{P}_q(I_n), n=1,\dots, N\}
\end{equation*}
for $q \in\mathbb{N}$. We denote by $\tjump{V}_{n}$ the time-jump across $t_n$, viz.,
\[
	\tjump{V}_n := V(t_n^+) - V(t_n^-) := \lim_{\epsilon \to 0^+} V(t_n+\epsilon) - \lim_{\epsilon \to 0^+} V(t_n-\epsilon)\quad\text{for } n=1,\dots,N.
\]

For later reference, we recall  {the standard} one-dimensional inverse estimates, see, e.g., \cite{schwab}: for $V\in \mathbb{P}_q(I_n)$, $n=1,\dots,N$, we have
\begin{equation}\label{time_trace_inv}
|V(t_{n}^-)| \leq (q+1)k_n^{-\frac{1}{2}}  \norm{V}{ I_n} \quad\text{ and }\quad \norm{V'}{I_n}\leq \sqrt{12}(q+1)^2k_n^{-1}  \|V\|_{ I_n}.
\end{equation}

\subsubsection{Fully-discrete scheme}
Starting from \eqref{def: seim-discrete bilinear form}, we further discretise the $t$-variable, resulting to the classical space-time dG method:  find $U \in  {\mathcal{V}}_{h,k}$ such that
\begin{equation}\label{eq:dG_FEM}
	B_{\rm st}(U,V)
	=
\ell_{\rm st}(V),\quad\text{ for all }\ V \in  {\mathcal{V}}_{h,k},
\end{equation}
 with
\begin{equation}\label{eq:space-time bilinear form}
  B_{\rm st}(U,V)
	:=
	\sum_{n=1}^N \int_{I_n}\!\! \big( \inner{U_t}{V}{} + a_{\rm dG}(U,V) \big) \dt + \sum_{n=2}^{N} \inner{\tjump{U}_{n-1}}{V( t_{n-1}^+)}{}
+	\inner{U(t_0^+)}{V(t_0^+)}{},
\end{equation}
and
\[
    \ell_{\rm st}(V)
    :=
    \sum_{n=1}^N \int_{I_n}
    \inner{f}{V}{}
    \dt
    +(U(t_0^-),V(t_0^+)),
    \]
where it is again convenient to choose $U(t_0^-) := \pi u_0$,  {as it can be implemented efficiently by inverting a block-diagonal mass matrix.} We note carefully that \eqref{eq:dG_FEM} is, again, the \emph{classical} upwind discontinuous Galerkin method of Houston, Schwab and S\"uli \cite{hss} for equations with non-negative characteristic form in three variables, (i.e., inclusive of $t$,) upon selecting a three-dimensional space-time mesh as above.

Of course, due to the ``space-time slab''-like selection of the mesh  {and the choice of the test function}, in practice \eqref{eq:dG_FEM} is typically solved on each time-step, viz.,
\[
        \int_{I_n}\! \! \big( \inner{U_t}{V}{} + a_{\rm dG}(U,V) \big)\! \dt +\inner{U(t_{n-1}^+)}{V( t_{n-1}^+)}{}
        =\! \int_{I_n}\! \inner{f}{V}{} \! \dt+\inner{U(t_{n-1}^-)}{V( t_{n-1}^+)}{},
\]
for all $V\in  {\mathcal{V}}_{h,k}$, $n=1,\dots, N$. Later, we shall make use of the following notation, for $n=1,\dots,N$:
\begin{equation}\label{B_n}
 B_n(U,V):= \int_{I_n}\! \! \big( \inner{U_t}{V}{} + a_{\rm dG}(U,V) \big)\! \dt +\inner{\tjump{U}_{n-1}}{V( t_{n-1}^+)}{}.
\end{equation}

The well-posedness of the fully-discrete problem follows from its coercivity; see \cite{hss}, or Theorem \ref{thm:discreteEquilibriation} below, for details.

\section{Asymptotic numerical hypocoercivity of the semi-discrete scheme}\label{sec:semi-discrete}
The strategy of the proof is to first demonstrate positivity in a stronger, hypo\-coercivity-inducing norm for the spatially discrete method \eqref{def: seim-discrete bilinear form}  by choosing optimal test functions of the form
\begin{equation}\label{semi: test function}
V: = U + \su\tau_T (x U_y + U_t)|_T + \su \big(-\Div (A\nabla U)|_T \big)\equiv V_1+ V_2 +V_3,
\end{equation}
with appropriate choices of $\tau$ defined in \eqref{def: tau}, and $\delta$ for $A$, defined in \eqref{def:delta_def} below, respectively.
\begin{remark}
A trivial, yet crucial aspect in the proof, is that $V$ is, indeed, admissible test function, since $V\in C^0(I; {\mathcal{V}}_h)$,  due to the nature of the first order transport term $x U_y + U_t$  which is element-wise a polynomial of \emph{total} order $p$ and, of course, due to the use of discontinuous approximation space. This is in sharp contrast with previous works \cite{georgoulis,DonGeoHer24}, whereby the spatial discretisation employs  standard, continuous, Lagrange elements, thereby requiring additional non-standard Petrov-Galerkin-type terms in the bilinear form, to ensure coercivity in a hypocoercivity-inducing norm.
\end{remark}

For the spatially discrete setting of this section, we set $\sigma:\gint\to\mathbb{R}$, defined face-wise by
\begin{equation}\label{eq:sigma_def}
    \sigma|_F :  = 64N_{\partial} (\cinvT n_1^Fp)^2 \max_{i=1,2} h_{T_i}^{-1}, \quad \text{for } F = \partial T_1 \cap \partial T_2,\ T_1,T_2 \in \Th,
\end{equation}
with $n_1^\omega:=\norm{n_1}{L_\infty(\omega)}$, for any $\omega\subset\gint$, and $N_\partial$ denoting the number of faces per element.  Also, we define $\tau:\Omega\to\mathbb{R}$ element-wise by
\begin{equation}\label{def: tau}
\begin{split}
    \tau|_{T}^{}=\tau_T^{}
    := &
    \frac{h_T^2}{p^4}\Big( \frac{65\cinv^2}{6}+\frac{3\sigma_T^2 h_T^2}{64p^4}+\frac{(8\xn\cinvT)^2 h_T}{3p^2}\Big)^{-1},
        \quad \text{for } T\in \Th,
    \end{split}
\end{equation}
with $\sigma_T: = \max_{F\in \partial T}{\sigma_F}$ and $\xn:=\norm{{xn_2}}{L_\infty(\partial_-T)}$, for notational brevity. We note that $\tau \sim h^2/p^4$ as  $h/p\to 0$, with $\sim$ denoting proportionality. In addition, we set
\begin{equation}\label{def:delta_def}\begin{split} \delta_T:=&\ \max\bigg\{1,\  \frac{\cinvT^2p^4}{h_T^2}\max\Big\{4+\frac{\cinv^2}{8
    	\cinvT^2},\frac{(\xn)^2h_T}{p^2},\sqrt{\mathcal{R}_T}\Big\}\bigg\}, \quad \text{for all $T\in \Th$}.
    \end{split}
\end{equation}
whereby
$\mathcal{R}_T:=2h_T^2 p^{-4} \tau_T^{-1}+1040N_\partial (n_1^FC_\rho\cinvT^{-1})^2+ 8(\cinv	\cinvT^{-1}\xn)^2p^{-2} h_T
$.
 Note that $\delta\ge 1$, which implies that $\alpha, \beta,\gamma<1$, and, also, we have
$\delta \sim p^4/h^2$,  as $h/p\to 0$

For later use we define the dG-seminorm:
\begin{equation}\label{dG:norm space}
\begin{split}
\norm{U}{\rm dG} &:=  \Big( \sum_{T \in \Th}\norm{U_x}{T}^2
+\frac{1}{2}\norm{U}{\rm uw}^2
    +
  \norm{\sqrt{\sigma} \jump{U}}{\gint }^2\Big)^{\frac{1}{2}},
\end{split}
\end{equation}
with the upwind terms collected in
\[
	\begin{split}
		\norm{U}{\rm uw} &:=  \Big( \sum_{T \in \Th}
		\big(\norm{ \sqrt{|x\normal_2|} U}{ (\partial_- T \cap \gif )\cup( \partial_+ T \cap \Gamma )}^2
+
		\norm{ \sqrt{|x\normal_2|} \jumpu{U}}{ \partial_- T \cap \gint }^2
	\big)\Big)^{\frac{1}{2}}.
	\end{split}
\]
The dG bilinear form $a_{\rm dG}(\cdot,\cdot)$ is coercive with respect to $\norm{\cdot}{\rm dG}$, see \cite{hss} or Lemma \ref{lem:semi:U} below for a proof. Crucially, however, $\norm{\cdot}{\rm dG}$ may fail to be a norm in $ {\mathcal{V}}_h$. Thus, the standard stability analysis in \cite{hss} does not apply.

The idea of the analysis below is to enhance $\norm{\cdot}{\rm dG}$ with additional terms to result into a norm in $ {\mathcal{V}}_h$, which will be shown to satisfy a discretisation parameter-dependent Poincar\'e inequality for functions in $ {\mathcal{V}}_h$ and, thus, we will eventually conclude the asymptotic numerical hypocoercivity property announced above. The proof will be split to a number of intermediate results, showing stability against each $V_i$, $i=1,2,3$, cf., \eqref{semi: test function} above, before collecting the bounds and completing the argument.
\begin{lemma}\label{lem:semi:U}
    Let $U \in H^1(I; {\mathcal{V}}_h)$ and let $\sigma$ as in \eqref{eq:sigma_def}. Then, a.e.~in $t$, we have
\[
   \inner{U_t}{V_1}{} + a_{\rm dG}(U,V_1)
        \geq
   \frac{1}{2}\frac{\dee}{\dt} \norm{U}{}^2
        +
      \frac{7}{8}\norm{U}{\rm dG}^2
         +\frac{1}{16}\norm{U}{\rm uw}^2.
\]
\end{lemma}
\begin{proof}
 {Recalling that $V_1=U$ and employing} the trace inverse estimate \eqref{trace_H1_inv},  we have,  {respectively,} %for $V_1=U$,
\begin{equation}\label{eq:standard_flux_estimate}
\begin{aligned}
 \inner{\avg{\normal_1 U_x}}{\jump{V_1}}{\gint }
 &\leq  \Big(\sum_{F\in \gint } \phi^{-1}(\cinvT n_1^Fp)^2 \sum_{T\in\omega_F}   h_T^{-1} \norm{  U_x}{ T}^2\Big)^{\frac{1}{2}} \norm{\sqrt{\phi} \jump{V_1}}{\gint} \\
 &\leq  \Big(\su  \norm{  U_x}{ T}^2 \Big)^{\frac{1}{2}} \norm{\sqrt{\phi} \jump{V_1}}{\gint} \\
&\leq \frac{1}{2\varepsilon}  \su \norm{ U_x}{ T}^2 +  \frac{\varepsilon}{2}\norm{\sqrt{\phi} \jump{V_1}}{\gint}^2,
\end{aligned}
\end{equation}
for $\varepsilon>0$, with $N_{\partial}$ denoting the number of faces per element,  $\omega_F$  is the collection of two elements sharing the face $F$, upon selecting $\phi|_F = N_{\partial}(\cinvT n_1^Fp)^2\max_{T\in \omega_F}h_T^{-1} $, when $n_1^F\neq 0$; when   $n_1^F=0$, no further estimation is required.

Now, integration by parts and investigation of the flux terms gives the identity
 \begin{equation}\label{eq:uw_id}
 	\begin{split}
		&\sum_{T \in \Th} \!\! \Big(\inner{x U_y}{U}{T}
		\!-\!
		\inner{x\normal_2 \jumpu{U}}{U^+}{\partial_- T \cap \gint}
	\!	+\!
		\inner{x\normal_2 U^+ }{ U^+ }{\partial_- T \cap \gif} \!\Big)\! = \frac{1}{2}\norm{U}{\rm uw}^2;
\end{split}\end{equation}
we refer to \cite[Lemma 2.4]{houston_schwab_suli_stabilised} for a detailed proof. Thus,  \eqref{eq:uw_id} yields
\[
\begin{split}
   a_{\rm dG}(U,U)
=
    \sum_{T \in \Th} \norm{U_x}{T}^2
    -
    2 \inner{\avg{\normal_1 U_x}}{\jump{U}}{\gint}
    +
    \norm{\sqrt{\sigma} \jump{U}}{\gint}^2
    +
\frac{1}{2}\norm{U}{\rm uw}^2.
\end{split}
\]
Employing \eqref{eq:standard_flux_estimate} with $\varepsilon=8$, along with the identity $ \inner{U_t}{U}{} =
\frac{1}{2}\frac{\dee}{\dt} \norm{U}{}^2$  and the splitting $\frac{1}{2}\norm{U}{\rm uw}^2=\big(\frac{7}{16}+\frac{1}{16}\big)\norm{U}{\rm uw}^2$, completes the proof,  {upon recalling the definition of $\sigma$.}
\end{proof}

\begin{lemma}\label{lem:semi:supg}
    Let $U \in H^1(I; {\mathcal{V}}_h)$ and $\tau$ as in \eqref{def: tau}. Then, a.e.~in $t$, we have
    \begin{equation}\label{eq:semi:supg}
        \inner{U_t}{V_2}{} + a_{\rm dG}(U,V_2)
        \geq
        \frac{1}{4} \sum_{T \in \Th} \norm{ \sqrt{\tau}(U_t + xU_y) }{T}^2 - \frac{1}{16} \norm{U}{\rm dG}^2.
    \end{equation}
\end{lemma}
\begin{proof}
Testing \eqref{def: seim-discrete bilinear form} against $V_2=\tau \utr$, with $\utr|_T:=(x U_y + U_t)|_T$, $T\in\Th$, and employing the inverse inequalities \eqref{trace_H1_inv}, along with standard estimation, gives
\begin{equation*}\begin{split}
    &  \inner{U_t}{\tau\utr}{} + a_{\rm dG}(U,\tau\utr) \\
   =&
    \sum_{T \in \Th}\!\big(
    \norm{\sqrt{\tau}\utr}{T}^2
    +
    \inner{U_x}{\tau\utr_x }{T}\!-\!
    \inner{x\normal_2 \jumpu{U}}{\tau W^+}{\partial_- T \cap \gint}
    \!-\!
    \inner{x\normal_2 U^+}{\tau W^+}{\partial_- T \cap \gif}\big)\\
   &\ -
    \inner{\avg{\normal_1 U_x}}{\jump{\tau\utr}}{\gint}
    -
    \inner{\avg{\normal_1 \tau\utr_x}}{\jump{U}}{\gint }
   +
    \inner{\sigma \jump{U}}{\jump{\tau\utr}}{\gint }.
\end{split}\end{equation*}
Employing inverse inequalities \eqref{trace_H1_inv} and standard steps, we estimate each group of terms separately, as follows. We have
\begin{equation*}\begin{split}
	&	  \sum_{T \in \Th}
		\inner{U_x}{\tau\utr_x }{T}
 \ge 	-\sum_{T \in \Th}
\frac{8 \cinv^2 p^4}{h_T^2}  \norm{\tau\utr }{T}^2 -\sum_{T \in \Th}\frac{1}{32}\norm{U_x}{T}^2.
\end{split}\end{equation*}
Also, working as in \eqref{eq:standard_flux_estimate} (for $\varepsilon=16$), and setting $\phi_T:=\max_{F\subset\partial T}\phi|_F$, with $\phi$ as in the proof of Lemma \ref{lem:semi:U}, we have
\begin{equation}\begin{split}\label{eq:double_inv_flux_estimate}
		\inner{\avg{\normal_1 U_x}}{\jump{\tau\utr}}{\gint}
		\ge&
		-
		\frac{1}{32}\su \norm{U_x}{T}^2 -8 \norm{\sqrt{\phi}\jump{\tau\utr}}{\gint}^2\\
		\ge &	-
		\frac{1}{32}\su \norm{U_x}{T}^2-16\sum_{F\in \gint }\! (\cinvT n_1^Fp)^2 \!\sum_{T\in\omega_F}  \! h_T^{-1} \norm{ \sqrt{\phi_T} \tau \utr}{ T}^2\\
			\ge &	-
		\frac{1}{32}\su \norm{U_x}{T}^2-16\su \phi_T^2\norm{  \tau \utr}{ T}^2,
\end{split}\end{equation}
 with an additional application of \eqref{trace_H1_inv}. Further, working as in \eqref{eq:standard_flux_estimate} (for $\varepsilon=4$),
 \[
 \begin{split}
 	\inner{\avg{\normal_1 \tau\utr_x}}{\jump{U}}{\gint }
 	\ge &
 		-
 	\frac{1}{8}\su \norm{ \tau\utr_x}{T}^2
 	-2 \norm{\sqrt{\phi}\jump{U}}{\gint}^2\\
 	\ge &
 	-
 	\su \frac{\cinv^2 p^4}{8h_T^2}\norm{ \tau\utr}{T}^2
 	-\frac{1}{32} \norm{\sqrt{\sigma}\jump{U}}{\gint}^2,
 \end{split}
 \]
 recalling that $\sigma=64\phi$. Also, working as before,
  \[
 \begin{split}
 	\inner{\sigma \jump{U}}{\jump{\tau\utr}}{\gint }
 		\ge &
 	-
 	8\norm{\sqrt{64\phi}\jump{\tau\utr}}{\gint}^2
 	-\frac{1}{32} \norm{\sqrt{\sigma}\jump{U}}{\gint}^2\\
 	\ge &
 	-
 	128\su \phi_T^2\norm{ \tau\utr}{T}^2
 	-\frac{1}{32} \norm{\sqrt{\sigma}\jump{U}}{\gint}^2.
 \end{split}
 \]
 Finally,
 \[
  \begin{aligned}
 &\su\Big(  \inner{x\normal_2 \jumpu{U}}{\tau W^+}{\partial_- T \cap \gint}
 \!+\!
 \inner{x\normal_2 U^+}{\tau W^+}{\partial_- T \cap \gif}\Big)
\\
\le&\ \frac{1}{32}\norm{U}{\rm uw}^2
 +16\su \!\frac{(\xn\cinvT  p)^2}{h_T} \norm{\tau\utr}{T}^2.
 \end{aligned}
 \]
Combining all the bounds and noting the definition of $\tau$ given in \eqref{def: tau}, we deduce
\begin{equation*}
\begin{split}
 &  \inner{U_t}{\tau\utr}{} + a_{\rm dG}(U,\tau\utr) \\
    \geq&
    \su\Big(
   \frac{1}{4} \norm{\sqrt{\tau} \utr }{T}^2
    -
     \frac{1}{16} \norm{U_x}{T}^2\Big)
 -
     \frac{1}{16} \norm{\sqrt{\sigma}\jump{U}}{\gint }^2
    -\frac{1}{32}\norm{U}{\rm uw}^2,
\end{split}\end{equation*}
from which the result follows.
\end{proof}
\begin{lemma}\label{lem:semi:hypo}
    Let $U \in H^1(I; {\mathcal{V}}_h)$, $A$ given by \eqref{eq:A_def} with $\alpha$, $\beta$, and $\gamma$ as in \eqref{eq:alpha_beta_gamma_def}, and $\delta$ as in \eqref{def:delta_def}. Then, a.e.~in $t$, we have the bound
    \begin{equation*}\begin{split}
        \inner{U_t}{V_3}{}  +    a_{\rm dG}(U,V_3)
        \geq&\,
        \frac{1}{2}  \sum_{T \in \Th}\! \Big(
       \frac{\dee}{\dt} \norm{\sqrt{A} \nabla U}{T}^2
        +
      \norm{\sqrt{A}\nabla U_x}{T}^2
        +
      \norm{ \sqrt{x\normal_2 A}\nabla U }{\partial_+ T}^2
        \\
        &
        \hspace{0cm}
        -
        \frac{1}{8} \norm{ \sqrt{\tau}(U_t + xU_y) }{T}^2
        +
        \frac{\alpha^2}{3}\norm{U_y}{T}^2
        -
        \frac{5}{4} \norm{U_x}{T}^2
        \Big)
        -
        \frac{1}{16}  \|U\|_{\rm dG}^2.
    \end{split}\end{equation*}
\end{lemma}
\begin{proof}

	Testing \eqref{def: seim-discrete bilinear form} with $V_3 = -\Div_h(A \nabla_h U)$, integrations by parts give
	\begin{equation*}\begin{split}
		&	  \inner{U_t}{V_3}{}  +    a_{\rm dG}(U,V_3)
			  \\
			    =&
			\sum_{T \in \Th} \bigg(
			\frac{1}{2} \frac{\dee}{\dt} \norm{\sqrt{A} \nabla U}{T}^2
			+
			\norm{\sqrt{A}\nabla U_x}{T}^2	+
			\inner{\nabla (x U_y)}{(A\nabla U)}{T}
			\\
			&\qquad\quad
			-
			\inner{U_t+ x U_y}{\normal \cdot (A\nabla U)}{\partial T}
			-
			\inner{U_x}{\normal \cdot (A\nabla U_x)}{\partial T}		\\
			&\qquad\quad
			+
			\inner{x\normal_2 \jumpu{U}}{ V_3^+ }{\partial_- T \cap \gint}
			+
			\inner{x\normal_2 U^+}{V_3^+}{\partial_- T \cap \gif}
			\bigg)
			\\
			&-
			\inner{\avg{\normal_1 U_x}}{\jump{ V_3}}{\gint }
			-
			\inner{\avg{\normal_1 \Div(A\nabla U_x) }}{\jump{U}}{\gint }
			+
			\inner{\sigma \jump{U}}{\jump{V_3}}{\gint }.
	\end{split}\end{equation*}
We estimate each term separately.

Focusing on the \emph{special} term $	\inner{\nabla(x U_y)}{A \nabla U}{T}$, we have, respectively,
	\begin{equation*}
		\begin{aligned}
	&	\inner{\nabla(x U_y)}{A \nabla U}{T}\\
		=&\
		 \alpha \inner{U_x}{U_y}{T}
		 +
		 \beta \norm{U_y}{T}^2
		 +
		 \inner{(x\sqrt{A} \nabla U)_y}{ \sqrt{A} \nabla U}{T}\\
		 =&\
		\alpha \inner{U_x}{U_y}{T}
		+
		\beta \norm{U_y}{T}^2
		+
		\frac{1}{2} \norm{ \sqrt{x\normal_2 A}\nabla U }{\partial_+ T}^2
		-
		\frac{1}{2} \norm{ \sqrt{|x\normal_2| A}\nabla U }{\partial_- T}^2\\
		\ge & \
	 \Big(\beta-\frac{2\alpha^2}{3}\Big)\norm{U_y}{T}^2	-	\frac{3}{8} \norm{U_x}{T} ^2+	\frac{1}{2} \norm{ \sqrt{x\normal_2 A}\nabla U }{\partial_+ T}^2
	 - \frac{(\xn\cinvT p)^2}{2h_T} \norm{\sqrt{A}\nabla U }{T}^2,
		\end{aligned}
	\end{equation*}
using the trace inverse inequality \eqref{trace_H1_inv}, with the understanding that $\alpha,\beta,\gamma$ are evaluated on $T$.

Now, from \eqref{trace_H1_inv}, along with elementary manipulation, we have
	\begin{equation*}\begin{split}
		-	   \inner{U_t+ x U_y}{\normal \cdot (A\nabla U)}{\partial T}
			\geq&\
		-	\frac{\cinvT ^2	p^2}{\sqrt{\tau_T}h_T} \norm{ \sqrt{\tau} (U_t+ x U_y)}{T}
	 \norm{A \nabla U}{ T}
			\\
			\geq &\
		-			\frac{1}{16}	\norm{ \sqrt{\tau} (U_t+ x U_y)}{T}^2 -  \frac{4\cinvT^4p^4}{\tau h_T^2}
			\norm{A \nabla U}{T}^2.
	\end{split}\end{equation*}
Working in a similar fashion, we also have
	\begin{equation*}
		-	\inner{U_x}{\normal \cdot (A\nabla U_x)}{\partial T}
			\geq
		-	\frac{1}{32}     \norm{U_x}{T}^2 -  \frac{ 8\cinvT^4p^4}{h_T^2}\norm{A\nabla U_x}{T}^2.
\end{equation*}
Next, working as in \eqref{eq:standard_flux_estimate} (for $\varepsilon=16$), and continuing as in \eqref{eq:double_inv_flux_estimate}, we have
	\begin{equation*}\begin{split}
			   \inner{\avg{\normal_1 U_x}}{\jump{V_3 }}{\gint}
			   \geq&\ - \frac{1}{32}  \su \norm{ U_x}{ T}^2 -  8\norm{\sqrt{\phi} \jump{V_3}}{\gint}^2\\
	   \geq&\ - \frac{1}{32}  \su \norm{ U_x}{ T}^2 -  16\su\phi_T^2\norm{V_3}{T}^2\\
	   \geq&\ - \frac{1}{32}  \su \norm{ U_x}{ T}^2 -  \su\frac{32\cinv^2\phi_T^2 p^4}{h_T^2}\norm{A\nabla U}{T}^2,
	\end{split}\end{equation*}
with an additional use of the inverse inequality \eqref{trace_H1_inv} in the last step.

Further, working as in \eqref{eq:standard_flux_estimate} (for $\varepsilon=4$) and recalling that $\sigma=64\phi$, we have
	\begin{equation*}\begin{split}
			 \inner{\avg{\normal_1 \Div(A\nabla U_x) }}{\jump{U}}{\gint}
			  \geq&\ -\frac{1}{8}  \su \norm{ \Div(A\nabla U_x) }{ T}^2 -  \frac{1}{32}\norm{\sqrt{\sigma} \jump{U}}{\gint}^2\\
			    \geq&\ -  \su \frac{\cinv^2p^4}{4h_T^2}\norm{ A\nabla U_x }{ T}^2 -  \frac{1}{32}\norm{\sqrt{\sigma} \jump{U}}{\gint}^2,
	\end{split}\end{equation*}
	and, similarly,
	\begin{equation*}\begin{split}
		-	   \inner{\sigma \jump{U}}{\jump{ V_3 }}{\gint }
			\geq& -\frac{1}{4}
			\su  \sigma_T^2\norm{ V_3}{T}^2
			-
			\frac{1}{32}   \norm{\sqrt{\sigma}\jump{U}}{\gint }^2\\
				\geq& -
			\su  \frac{\cinv^2\sigma_T^2p^4}{2h_T^2}\norm{A\nabla U}{T}^2
			-
			\frac{1}{32}   \norm{\sqrt{\sigma}\jump{U}}{\gint }^2.
	\end{split}\end{equation*}
	Finally, we also have
	\begin{equation*}\begin{split}
			&\su \Big( \inner{x\normal_2 \jumpu{U}}{ \Div(A\nabla U^+) }{\partial_- T \cap \gint}
			+
			\inner{x\normal_2 U^+}{ \Div(A\nabla U^+) }{\partial_- T \cap \gif}\Big)
			\\
			&\geq  -\sqrt{2}\norm{U}{\rm uw}
	\su	\frac{	\cinvT\cinv	\xn p^3}{h_T^{\frac{3}{2}}}
			\norm{ A \nabla U }{T}\\
			&
			\geq
			-\frac{1}{32}\norm{U}{\rm uw}^2
			-
			\su\frac{(4\cinvT\cinv	\xn p^3)^2}{h_T^3}
			\norm{ A\nabla U }{T}^2.
	\end{split}\end{equation*}
	For convenience, we denote by $|A|$ the spectral norm of $A$, i.e., its largest eigenvalue, and we note the elementary estimate $|A| \leq 2 \alpha=(4\delta)^{-1}$. The proof is presented in Lemma \ref{eigen_est}. Also, observing that $\max_{i=1,2}h_i^{-1}\le C_{\rho} h_T^{-1}$, and that $\sigma_T=64\phi_T$, we have
	\[
	\frac{\cinv^2\sigma_T^2p^4}{2h_T^2}+  \frac{32\cinv^2\phi_T^2 p^4}{h_T^2}\le \frac{2080N_\partial (n_1^FC_\rho\cinvT)^2p^8}{h_T^4},
	\]
	so that, the definition of $\delta$, \eqref{def:delta_def}, implies
	\[
	\Big( \Big(	   \frac{4\cinvT^4p^4}{\tau h_T^2}+ \frac{2080N_\partial (n_1^FC_\rho\cinvT)^2p^8}{h_T^4}+ \frac{(4\cinvT\cinv	\xn p^3)^2}{h_T^3} \Big)|A|_T| + \frac{(\xn\cinvT p)^2}{2h_T}\Big)\le \delta_T,
	\]
	and also
	$
	\big(1 - \big( 8\cinvT^4+   \frac{\cinv^2}{4} \big) \frac{| A|_T |p^4}{h_T^2}\big) \ge \frac{1}{2}$.

	Combining the above bounds, we arrive at
\begin{equation}\label{eq:inder_Vthree}
	\begin{aligned}
			&     \inner{U_t}{V_3)}{}  +    a_{\rm dG}(U,V_3) \\
		\geq
	&		\sum_{T \in \Th} \bigg(
		\frac{1}{2} \frac{\dee}{\dt} \norm{\sqrt{A} \nabla U}{T}^2
		+
	\frac{1}{2}  \norm{\sqrt{A}\nabla U_x}{T}^2+	\frac{1}{2} \norm{ \sqrt{x\normal_2 A}\nabla U }{\partial_+ T}^2- \delta_T\norm{\sqrt{A}\nabla U}{T}^2
		\\
	&\qquad	-			\frac{1}{16}	\norm{ \sqrt{\tau} (U_t+ x U_y)}{T}^2
		+ \Big(\beta-\frac{2\alpha^2}{3}\Big)\norm{U_y}{T}^2	-	\frac{3}{8} \norm{U_x}{T} ^2\bigg)	-\frac{1}{16}\norm{U}{\rm dG}^2.
	\end{aligned}
\end{equation}
Also, since
\[
\|\sqrt{A}\nabla U\|_T^2= \alpha \|U_x\|_T^2+ 2\beta (U_x, U_y)_T + \gamma \|U_y\|_T^2\le  2\alpha \|U_x\|_T^2+\Big( \gamma + \frac{\beta^2}{\alpha }\Big) \|U_y\|_T^2,
\]
 {using the above bound, we infer
\[
	\begin{aligned}
			&     \inner{U_t}{V_3)}{}  +    a_{\rm dG}(U,V_3) \\
		\geq
	&		\sum_{T \in \Th} \bigg(
		\frac{1}{2} \frac{\dee}{\dt} \norm{\sqrt{A} \nabla U}{T}^2
		+
	\frac{1}{2}  \norm{\sqrt{A}\nabla U_x}{T}^2+	\frac{1}{2} \norm{ \sqrt{x\normal_2 A}\nabla U }{\partial_+ T}^2- \delta_T\norm{\sqrt{A}\nabla U}{T}^2
		\\
	&\quad	-			\frac{1}{16}	\norm{ \sqrt{\tau} (U_t+ x U_y)}{T}^2
		+ \Big(	\beta -\frac{2\alpha^2}{3} - \frac{\delta_T \beta^2}{\alpha} -\delta_T \gamma\Big)\norm{U_y}{T}^2
\\	&\quad 	-	(\frac{3}{8} +2\delta_T \alpha) \norm{U_x}{T} ^2\bigg)
-\frac{1}{16}\norm{U}{\rm dG}^2,
	\end{aligned}
\]}
and, 	recalling $\alpha = (8\delta)^{-1}$, $\beta = (48\delta^2)^{-1}$ and $\gamma = 10(48^2\delta^3)^{-1}$, we calculate
	\begin{equation*}\begin{split}
		\beta -\frac{2\alpha^2}{3} - \frac{\delta_T \beta^2}{\alpha} -\delta_T \gamma
		=
		(384 \delta_T^2)^{-1} =&\ \frac{\alpha^2}{6}, \quad\text{and}\quad
		 \frac{3}{8 }  +2\delta_T \alpha =\ \frac{5}{8},
\end{split}\end{equation*}
on each $T\in\Th$.
The result follows upon using the last bound and the last two identities on \eqref{eq:inder_Vthree}.
\end{proof}
The next step is to combine Lemmata \ref{lem:semi:U}, \ref{lem:semi:supg}, and \ref{lem:semi:hypo}, to show a positivity result, which will eventually imply the asymptotic numerical hypocoercivity of the semi-discrete scheme.
To that end, we define the ``enhanced'' norm
\begin{equation}\label{def: semi-discrete dG norm}
\begin{split}
\ndg{w}: =&\Big(
    \sum_{T \in \Th} \Big(
    \norm{\sqrt{A}\nabla w_x}{T}^2
  +
     \frac{1}{2}  \norm{\alpha w_y}{T}^2
+
  \frac{1}{2} \norm{w_x}{T}^2 +
  \frac{1}{2 } \norm{ \sqrt{\tau}(w_t + xw_y) }{T}^2\\
 &   \hspace{1.2cm}
+
\norm{ \sqrt{x\normal_2 A}\nabla w }{\partial_+ T}^2
\Big)
+ \norm{\sqrt{\sigma} \jump{w}}{\gint }^2+2\norm{w}{\rm uw}^2\Big)^{\frac{1}{2}},
\end{split}
\end{equation}
for all $w\in H^2(\Omega,\Th)$,
along with the ``broken'' $A$-inner product and norm:
\[
\Ahinner{w}{v}:= (w,v)+(A\nabla_h w,\nabla_h v), \quad\text{ and }\quad \norm{w}{A,h}:=\sqrt{\Ahinner{w}{w}},
\]
for all $w,v\in H^1(\Omega,\Th)$.
\begin{proposition}\label{stab_semi}
    Let $U \in H^1(I; {\mathcal{V}}_h)$, $V\in  {\mathcal{V}}_h$ given by \eqref{semi: test function}, $A$ given by \eqref{eq:A_def} with $\alpha$, $\beta$, and $\gamma$ as in \eqref{eq:alpha_beta_gamma_def}, and $\delta$ as in \eqref{def:delta_def}. Then, a.e.~in $t$, we have the estimate
    \begin{equation}
        \inner{U_t}{V}{}  +  a_{\rm dG}(U,V)
        \geq
        \frac{1}{2} \frac{\dee}{\dt}\norm{U}{A,h}^2+ \frac{1}{4}\ndg{U}^2.
    \end{equation}
\end{proposition}
\begin{proof}
	The result follows by combining the bounds in Lemmata \ref{lem:semi:U}, \ref{lem:semi:supg}, and \ref{lem:semi:hypo}.
\end{proof}

The next step is to show that $\ndg{\cdot}$ controls $\norm{\cdot}{A}$, that is we have ``Poincar\'e inequality''/``spectral gap''. Due to the dependence of $A$ on the mesh parameters $h$ and $p$,  as expected, we can only show a \emph{mesh dependent spectral gap}, with the respective constant degenerating as $h\to 0$ and/or $p\to\infty$.
\begin{lemma}[mesh-dependent spectral gap]\label{lemma:mesh-dependent spectral gap}
    There exists a number $\kappa\equiv\kappa(h,p)>0$, depending  {on} $h_{\min}:=\min_{\Omega}h$, on the polynomial degree $p$, on the domain geometry, and on the mesh shape-regularity constant $C_\rho$, such that, for all $w \in  {\mathcal{V}}_{h}$, we have
    \begin{equation}\label{mesh-dependent spectral gap}
    \ndg{w}^2 \geq \kappa\, \norm{w}{A,h}^2.
    \end{equation}
    Moreover, as $h_{\min}/ p\to 0$, we have $\kappa\equiv \kappa(h,p)\sim h^4_{\min} p^{-8}$.
    \end{lemma}
    \begin{proof}
    Let $B:= {\rm diag}(1,\alpha^2)$, recalling that $\alpha=(8\delta)^{-1}$.  We employ the splitting
    \[	\norm{\sqrt{B}\nabla_h w}{}^2
    = \norm{\sqrt{B-\nu A}\nabla_h w}{}^2
    +\norm{\sqrt{\nu A}\nabla _hw}{}^2,
    \]
    for $\nu=(2\delta)^{-1}$. Then, $B-\nu A$ is positive definite with smallest eigenvalue satisfying $\lambda_{1}^{B-\nu A}\ge \frac{\det(B-\nu A)}{1-\nu\alpha + |\nu\beta|}\ge (228\delta^2)^{-1}$; cf., Lemma \ref{eigen_est}.

    Since for a face $F=\partial T_1\cap\partial T_2$ shared by two elements $T_1,T_2\in \Th$, we have $\max_{i=1,2}h_i^{-1}\ge (\avg{h}|_F)^{-1}$, and since $64N_\partial p^2 \ge 1$, we already have the lower bound
    \begin{equation}\label{lower_all}
    \begin{aligned}
 2	\ndg{w}^2\ge &\   \min_\Omega\frac{1}{228\delta^2}  \norm{\nabla_h w}{}^2+\norm{\sqrt{\nu A}\nabla _hw}{}^2 +\norm{n_1\avg{h}^{-\frac{1}{2}}\jump{w}}{\gint }^2\\
& +\su\norm{\sqrt{n_2 x}\jump{w}}{\partial_-T\cap \gint }^2+   \norm{\sqrt{|n_2 x|}\jump{w}}{ \gif\cup\Gamma_+ }^2.
 \end{aligned}
    \end{equation}

To conclude the proof it is sufficient to find suitable $ L(h,p)>0$, such that
\[
\begin{aligned}
&\su \Big( \norm{n_1\avg{h}^{-\frac{1}{2}}\jump{w}}{\partial T\cap \gint }^2 +\norm{\sqrt{n_2 x}\jump{w}}{\partial_-T\cap \gint }^2\Big)+   \norm{\sqrt{|n_2 x|}w}{\gif\cup\Gamma_+ }^2\\
\ge &\  L(h,p) \Big(\norm{\avg{h}^{-\frac{1}{2}}\jump{w}}{ \gint }^2 +\norm{h^{-\frac{1}{2}}w}{ \gif\cup\Gamma_+ }^2 \Big).
\end{aligned}
\]
To that end, we first observe that the left-hand side sum runs through all the internal faces in the mesh. Indeed, each element face $e$ in $\gint$ is either inflow or outflow. If it is outflow, for a given element, it is inflow of the element ``downstream''.

First consider a face $e\subset\gint$, with $|n_2|< \frac{1}{2}$, i.e., a ``vertical-like'' face. Then, $n_1^2=1-n_2^2\ge \frac{3}{4}$ and, so, $\norm{n_1\avg{h}^{-\frac{1}{2}}\jump{w}}{e}^2\ge \frac{3}{4} \norm{\avg{h}^{-\frac{1}{2}}\jump{w}}{e}^2 $.

 Now, we consider  ``horizontal-like'' faces for which $|n_2|\ge \frac{1}{2
}$.
We set $\barhmin:=\frac{1}{2}\min\{1, h_{\min}\}$, and we separate two cases: a) $|x|\ge2\barhmin$ for all $(x,y)\in e$, i.e., ``sufficiently far'' from the $y$-axis, and b) $|x|< 2\barhmin$ for some $(x,y)\in e$, that is ``very close'' or crossing the $y$-axis.

For a), we have $|xn_2|\ge \barhmin$, to conclude $\norm{\sqrt{|n_2 x|}\jump{w}}{e }^2\ge 2 \barhmin^2\norm{\avg{h}^{-\frac{1}{2}}\jump{w}}{e}^2$.

For b), (that is, $|x|< 2\barhmin$ for some $(x,y)\in e$ and $|n_2|\ge \frac{1}{2}$, i.e., $e$ is a ``horizontal-like'' face and is ``very close'' or crosses the $y$ axis,) we begin by considering the near-most piece of $e$ to the $y$-axis, denoted by $e_0$,
 for which
$ |x|
\le  \barhmin(256p^2)^{-1}$, for all $(x,y)\in e_0$. (If it so happens that $ |x|
\ge  \barhmin(256p^2)^{-1}$, for all $(x,y)\in e$, then $e=e_0$.)

Then, we have
\begin{equation}\label{lower_b}
\norm{\sqrt{|n_2 x|}\jump{w}}{e}^2\ge  \norm{\sqrt{|n_2 x|}\jump{w}}{e\backslash e_0}^2 \ge \frac{\barhmin}{512p^2}\norm{\jump{w}}{e\backslash e_0}^2.
\end{equation}
Further, recalling the one-dimensional inverse estimate $\norm{v}{L_\infty(e)}^2\le 32p^2|e|^{-1}\norm{v}{L_2(e)}^2$, for $v$ polynomial of degree $p$ on $e$, (see, \cite{MR839312} or, e.g., \cite[Eq.~(3.6.4)]{schwab},) we have
\[
\begin{aligned}
\norm{\jump{w}}{e\backslash e_0}^2
= &\
\norm{\jump{w}}{e}^2-\norm{\jump{w}}{e_0}^2
\ge
\norm{\jump{w}}{e}^2-|e_0|\norm{\jump{w}}{L_\infty(e)}^2\\
\ge &\
\norm{\jump{w}}{e}^2-|e_0|32p^2 |e|^{-1}\norm{\jump{w}}{e}^2,
\end{aligned}
\]
respectively.  Upon noting that  $|e_0|\le \barhmin(64p^2)^{-1}$ by construction since $n_2\ge \frac{1}{2}$.
 Combining the last two bounds, we conclude $\norm{\jump{w}}{e\backslash e_0}^2 \ge \frac{1}{2}\norm{\jump{w}}{e}^2$, which in conjunction with \eqref{lower_b}, gives
 \[
 \norm{\sqrt{|n_2 x|}\jump{w}}{e}^2\ge \frac{\barhmin}{1024 p^2}\norm{\jump{w}}{e}^2\ge \frac{\barhmin^2}{1024 p^2}\norm{\avg{h}^{-\frac{1}{2}}\jump{w}}{e}^2.
 \]
 Note that the above argument remains valid also for boundary faces $e\subset \gif\cup\Gamma_+$.

  From the above, it is sufficient to select
$
 L(h,p):= \barhmin^2(1024p^2)^{-1}
$,
upon noting the bound $\barhmin^2(1024p^2)^{-1}\le \max\{\frac{3}{4},2\barhmin^2\}$. Thus, \eqref{lower_all} gives
\[
\begin{aligned}
2	\ndg{w}^2\ge &\   \min\Big \{\min_\Omega\frac{1}{228\delta^2},  \frac{\barhmin^2}{1024p^2}\Big\} \Big( \norm{\nabla_h w}{}^2 +\norm{\avg{h}^{-\frac{1}{2}}\jump{w}}{\gint \cup\gif\cup\Gamma_+}^2 \Big)\\
&+\norm{\sqrt{\nu A}\nabla _hw}{}^2
\end{aligned}
\]

Further, it is known (see  {\cite[Theorem 1.6]{botti2025sobolev} for details}) that the ``broken''  {Poincar\'e} inequality (spectral gap)  holds in this case:
\begin{equation}\label{broken_poinc}
\norm{w}{}^2 \leq C_{\rm bPF} \big(  \norm{\nabla_h w}{}^2 +\norm{\avg{h}^{-\frac{1}{2}}\jump{w}}{\gint \cup\gif\cup\Gamma_+}^2 \big),
\end{equation}
for a constant $C_{\rm bPF}>0$, independent of $w$ and of the discretisation parameters, but dependent on the length of $\gif\cup\Gamma_+$. Then, \eqref{mesh-dependent spectral gap} follows by setting
\[
\kappa\equiv\kappa(h,p):=\frac{1}{2C_{\rm bPF}}  \min\Big \{\min_\Omega\frac{1}{228\delta^2},  \frac{\barhmin^2}{1024p^2}\Big\}.
\]

Thus, as $h_{\min}/p\to 0$, we conclude $\kappa= \min_{\Omega}(C_{\rm bPF} 228\delta^{2})^{-1}\sim  h_{\min}^4 p^{-8}$.
    \end{proof}

 {
\begin{remark}
Note that $\ndg{\cdot}$ is, indeed, a norm in $\mathcal{V}_{h}$ since it majorises the norm $\norm{\cdot}{A,h}$.
\end{remark}
}

We are now ready to prove the first main result of this work.
\begin{theorem}[Asymptotic numerical hypocoercivity of the semi-discrete scheme]\label{main_SD}
    Let $U \in H^1(I; {\mathcal{V}}_h)$ be the solution of the semi-discrete scheme \eqref{def: seim-discrete bilinear form} for $f=0$. Then, we have
    \begin{equation}\label{hypoSUPG_stab_semidsicrete}
    \norm{U(t_f)}{ {A,h}}\le e^{-\kappa t_f/4} \norm{U(0)}{ {A,h}},
    \end{equation}
 {where $\kappa\equiv \kappa(h,p)\sim h^4_{\min} p^{-8}$, as $h_{\min}/ p\to 0$}.
\end{theorem}
\begin{proof}
	From Lemma \ref{lemma:mesh-dependent spectral gap} and Proposition \ref{stab_semi}, we have, respectively,
	\begin{equation}\label{eq: semi-discrete dG hypocoercivity}
		\begin{aligned}
			\frac{1}{2}\frac{\dee}{\dt}\norm{U}{ {A,h}}^2+ \frac{\kappa}{4} \norm{U}{ {A,h}}^2\le 		\frac{1}{2}\frac{\dee}{\dt}\norm{U}{ {A,h}}^2+	\frac{1}{4}\ndg{U}^2
			\le   \inner{U_t}{V}{}  + a(U,V)= 0.
		\end{aligned}
	\end{equation}
	The result follows by integration with respect to $t$ in the interval $I$, noting that $U(0)\in  {\mathcal{V}}_h\subset H^1(\Omega,\Th)$.
\end{proof}
The above result shows that the classical discontinuous Galerkin method retains an ``asymptotically preserving'' hypocoercivity structure at the discrete level, in the sense that the spectral gap becomes significant for $t_f>\kappa^{-1}\sim p^8h_{\min}^{-4}$. This is, perhaps, expected, due to the bounded domain and the boundary conditions. We refer to the preservation of this structure by a numerical method as \emph{numerical hypocoercivity}.

\begin{remark}\label{remark_stab}
To highlight the importance of $V_3$ in the present context, we discuss briefly the effect of omitting this term. Indeed, upon testing with $V_1+V_2$ only, we return to the inf-sup stability analysis of the classical dG method in a stronger-than-energy, SUPG-like norm; we refer to \cite{ayuso_marini,cdgh16,MR3672375}, for various, related results in this context, extending the seminal, classical result presented in \cite{johnson_pitkaranta}. Combining the bounds in the bounds in Lemmata \ref{lem:semi:U}, \ref{lem:semi:supg} only, we deduce
  \begin{equation*}
	\inner{U_t}{V_1+V_2}{} + a_{\rm dG}(U,V_1+V_2)
	\geq
	\frac{1}{2}\frac{\dee}{\dt} \norm{U}{}^2
	+
	\frac{7}{8}\norm{U}{\rm dG}^2
+
	\frac{1}{4} \sum_{T \in \Th} \norm{ \sqrt{\tau}(U_t + xU_y) }{T}^2.
\end{equation*}
Now,  {it is not known if it is possible to bound $\big(\norm{U}{\rm dG}^2
+
\frac{1}{4} \sum_{T \in \Th} \norm{ \sqrt{\tau}(U_t + xU_y) }{T}^2\big)^{1/2}$ by $\norm{U}{}$ from below in general},  cf., \eqref{broken_poinc}, given that $\|U_y\|$ is \emph{not} present in $\norm{U}{\rm dG}^2$. Thus, the energy argument for $f=0$ becomes
\[
	\frac{1}{2}\frac{\dee}{\dt}\norm{U}{}^2
	\le   \inner{U_t}{V_1+V_2}{}  + a(U,V_1+V_2)= 0,
\]
which, in turn, implies just $ \norm{U(t_f)}{}\le  \norm{U(0)}{}$; cf., \eqref{hypoSUPG_stab_semidsicrete}.
\end{remark}

\section{Asymptotic numerical hypocoercivity of the space-time method}\label{sec:fullDiscrete}
The respective result for the fully-discrete scheme \eqref{eq:dG_FEM} providing the space-time discrete solution $U\in  {\mathcal{V}}_{h,k}$, follows by extending the results of the previous section. To that end, we consider the test function $V$,  {which is defined in the fully-discrete space-time dG space $\mathcal{V}_{h,k}$}, along with the $V_i$, $i=1,2,3$, as:
\begin{equation}\label{FDV}
V: = U + \sum_{I_n \in \mathcal{I}} \su\tau_{T,n} (x U_y + U_t)|_{I_n\times T}
+ \sum_{I_n \in \mathcal{I}} \su -\Div(A \nabla U)|_{I_n\times T} \equiv V_1+ V_2 +V_3,
\end{equation}
upon redefining $\tau$ and $\delta$ as follows. We set $\tau|_{I_n\times T}:= \tau_{T,n}$, with
\begin{equation}\label{def: tau_st}
\begin{split}
 \tau_{T,n}
    := &
    \min \Big\{\tau_T,
 \frac{k_n}{64(q+1)^2}
        \Big\},
    \end{split}
\end{equation}
for all $T \in \Th$, $n = 1,\dots, N$, with $\tau_T$ as in \eqref{def: tau}. Note that  $\tau \sim \min\{ h^2/p^4, k/q^2\}$,  as $h/p\to 0$ and as $k/q\to 0$, with $k=\max_{n=1,\dots,N} k_n$.

Also, we modify (slightly) the definition of the (space-time) element-wise constant function $\delta$, which is now given by $\delta|_{I_n\times T}:=\delta_{T,n}$, $T\in\Th$, $n=1,\dots, N$, where
\begin{equation}\label{def:deltaFD_def}
	\begin{split} \delta_{T,n}:=&\ \max\bigg\{1,\  \frac{\cinvT^2p^4}{h_T^2}\max\Big\{4+\frac{\cinv^2}{8
			\cinvT^2},\frac{(\xn)^2h_T}{p^2},\sqrt{\mathcal{R}_{T,n}}\Big\}\bigg\},
	\end{split}
\end{equation}
with
\[
\begin{aligned}
\mathcal{R}_{T,n}:=&\ 2h_T^2 p^{-4} \tau_{T,n}^{-1}+1040N_\partial (n_1^FC_\rho\cinvT^{-1})^2\\
&+ 8(\cinv	\cinvT^{-1}\xn)^2p^{-2} h_T+2 h_T^2 (q+1)^2p^{-4} k_n^{-1}.
\end{aligned}
\]
Note that we have
$\delta \sim \max\{ p^4/h^2, p^2q/(h\sqrt{k})\}$,  as $h/p\to 0$ and as $k/q\to 0$. In particular, for later reference, we have $\delta_{T,n}\ge 1$ and
\[
\delta_{T,n}\ge \frac{\cinvT^2p^4}{h_T^2}\sqrt{\mathcal{R}_{T,n}}\ge \frac{\sqrt{2}\cinvT^2p^2(q+1)}{h_T\sqrt{k_n}},
\]
and, so, on each $I_n\times T$, we have
\begin{equation}\label{final_bound_alpha}
\alpha=(8\delta_{T,n})^{-1} <\min\Big\{ 1,  \frac{h_T\sqrt{k_n}}{\cinvT^2p^2(q+1)}\Big\}.
\end{equation}

\subsection{Stability estimates}
We prove lower bounds for $B_n$, from \eqref{B_n}, when tested against each one of $V_i$, $i=1,2,3$.
\begin{lemma}\label{eq:energy_st}
    Let $U \in  {\mathcal{V}}_{h,k}$. Then, for $\sigma$ as given in \eqref{eq:sigma_def}, we have
    \begin{equation*}
        \begin{split}
            B_n(U,V_1)
            \geq &\  \frac{1}{2} \left( \norm{\tjump{U}_{n-1}}{}^2 + \norm{U(t_n^-)}{}^2  - \norm{U(t_{n-1}^-)}{}^2\right)
            \\
            &+\int_{I_n}   \big(
            \frac{7}{8}\norm{U}{\rm dG}^2
            +\frac{1}{16}\norm{U}{\rm uw}^2\big) \dt .
        \end{split}\end{equation*}
\end{lemma}
\begin{proof}
    Since $V_1=U$, $B_n(U,U)$ includes $\int_{I_n}a_{\rm dG}(U,U) \dt$, which has essentially been treated in Lemma \ref{lem:semi:U}.
    For the remaining, time discretisation terms, we have
    \begin{equation*}
        \int_{I_n} \inner{U_t}{U}{}\dt
        +
        \inner{\tjump{U}_{n-1}}{U(t_{n-1}^+)}{}
        =
        \frac{1}{2} \left( \norm{\tjump{U}_{n-1}}{}^2 + \norm{U(t_n^-)}{}^2 - \norm{U(t_{n-1}^-)}{}^2 \right),
    \end{equation*}
and the result already follows.
\end{proof}
\begin{lemma}\label{eq:SUPG_st}
    Let $U \in  {\mathcal{V}}_{h,k}$ and $\tau$ as in \eqref{def: tau_st}. Then, on each $I_n$, we have
    \begin{equation}
        B_n(U,V_2)
        \geq
        \int_{I_n} \Big( \frac{3}{16} \sum_{T \in \Th} \norm{ \sqrt{\tau}(U_t + xU_y) }{T}^2 - \frac{1}{16} \norm{U}{\rm dG}^2 \Big) \dt
        - \frac{1}{16}\norm{\tjump{U}_{n-1}}{}^2.
    \end{equation}
\end{lemma}

\begin{proof}
Test $B_n(U,\cdot)$ against $V_2=\tau \utr$, with $\utr|_{I_n\times T}:=(x U_y + U_t)|_{I_n\times T}$, $n=1,\dots, N$, $T\in\Th$, (noting that $W$ is now a space-time polynomial on each space-time element $I_n\times T$,) with
 $\tau$ as given in \eqref{def: tau_st}. Note that this way we retain the validity of the bound in Lemma \ref{lem:semi:supg} and, thus,
    \begin{equation*}
\begin{split}
        B_n(U,V_2)
        \geq&
        \int_{I_n} \Big(
            \frac{1}{4} \sum_{T \in \Th} \norm{ \sqrt{\tau}W }{T}^2
            - \frac{1}{16} \|U\|_{dG}^2
        \Big) \dt
        +
        \inner{\tjump{U}_{n-1}}{\tau W( t_{n-1}^+)}{}.
\end{split}
    \end{equation*}
    Employing \eqref{time_trace_inv}, and standard estimation steps, we have
    \begin{equation*}\begin{split}
        \inner{\tjump{U}_{n-1}}{V_2( t_{n-1}^+)}{}
        &
         \geq
        -\su 4 \tau_{T,n} (q+1)^2k_n^{-1}\int_{I_n} \norm{ \sqrt{\tau}W}{T}^2
        \dt
        - \frac{1}{16} \norm{\tjump{U}_{n-1}}{}^2.
    \end{split}\end{equation*}
    Since $ 64\tau_{T,n} (q+1)^2k_n^{-1} \leq 1$, due to the  choice of $\tau$ in \eqref{def: tau_st}, the result follows.
\end{proof}

\begin{lemma} \label{eq:hypo_st}
    Let $U \in  {\mathcal{V}}_{h,k}$ and $\delta$ be as given in \eqref{def:deltaFD_def}. Then, on each $I_n$, we have
    \begin{equation*}
        \begin{split}
            B_n(U,V_3)\ge &\
            \frac{1}{2}\Big( \norm{\tjump{\sqrt{A}\nabla_h U}_{n-1}}{}^2
            \!+\!	\norm{\sqrt{A}\nabla_h U(t_n^-)}{}^2 \! -\!	\norm{\sqrt{A}\nabla_h U(t_{n-1}^-)}{}^2   \Big)\\
            &	- \frac{1}{16} \norm{\tjump{U}_{n-1}}{}^2 +\int_{I_n}\!\! \bigg(\frac{1}{2}
                \sum_{T \in \Th} \Big(
             \norm{\sqrt{A}\nabla U_x}{T}^2
            +
               \norm{ \sqrt{x\normal_2 A}\nabla U }{\partial_+ T}^2
                \\
              &
             -
            \frac{1}{8} \norm{ \sqrt{\tau}(U_t + xU_y) }{T}^2
                +
                \frac{\alpha^2}{3} \norm{U_y}{T}^2
            -
                \frac{5}{4} \norm{U_x}{T}^2
            \Big)
            \!-\!
            \frac{1}{16}  \|U\|_{\rm dG}^2 \bigg)\!\ud t .
        \end{split}\end{equation*}
\end{lemma}
\begin{proof}
On each $I_n$, we have
\begin{equation*}\begin{split}
		B_n(U,V_3)
		=&
		-\int_{I_n} \left(
		\inner{U_t}{ \Div_h ( A \nabla_h U)}{}
		+
		a_{\rm dG}(U, \Div_h ( A \nabla_h U))
		\right) \dt
		\\
		&
		- \inner{\tjump{U}_{n-1}}{\Div_h ( A \nabla_h U)( t_{n-1}^+)}{}.
\end{split}\end{equation*}
The first integral on the right-hand side of the last equality can be handled completely analogously to Lemma \ref{lem:semi:hypo}; we, thus, focus on the second term. Integrating by parts in space on each $T\in \Th$, we have
\begin{equation*}
	\begin{split}
		- \inner{\tjump{U}_{n-1}}{\Div_h ( A \nabla_h U)( t_{n-1}^+)}{}
		=&\
		\inner{\tjump{\nabla U}_{n-1}}{ ( A \nabla_h U)( t_{n-1}^+)}{} \\
		&   - \su \inner{\tjump{U}_{n-1}}{\normal \cdot ( A \nabla_h U)( t_{n-1}^+)}{\partial T}.
	\end{split}
\end{equation*}
Using the inverse inequalities \eqref{trace_H1_inv} and \eqref{time_trace_inv},  we have
\begin{equation}\label{eq: relation 1}
	\begin{split}
		& - \inner{\tjump{U}_{n-1}}{\normal \cdot ( A \nabla_h U)( t_{n-1}^+)}{\partial T}
		\\
		& \geq
		-  \frac{(\cinvT p)^2 (q+1)|\sqrt{A}|_T|}{h_T \sqrt{k_n}} \norm{\tjump{U}_{n-1}}{T} \norm{ \sqrt{A} \nabla U}{I_n\times T}  \\
		& \geq
		-   \frac{4 (\cinvT p)^4(q+1)^2|A|_T|}{h_T^2 k_n} \norm{ \sqrt{A} \nabla U}{I_n\times T}^2
		- \frac{1}{16} \norm{\tjump{U}_{n-1}}{T}^2.
\end{split}\end{equation}
The following identity holds on each time interval $I_n$:
\begin{equation}\label{eq: relation 2}
	\begin{split}
		&  \int_{I_n} \left( \inner{A \nabla_h U_t}{  \nabla_h U}{} \right) \dt
		+ \inner{\tjump{\nabla_h U}_{n-1}}{ ( A \nabla_h U)( t_{n-1}^+)}{}
		\\
		&= \frac{1}{2}\left( \norm{\tjump{\sqrt{A}\nabla_h U}_{n-1}}{}^2
		+	\norm{\sqrt{A}\nabla_h U(t_n^-)}{}^2 -	\norm{\sqrt{A}\nabla_h U(t_{n-1}^-)}{}^2   \right);
\end{split}\end{equation}
its proof follows upon observing $ \inner{A \nabla_h U_t}{  \nabla_h U}{}=\frac{1}{2} \frac{\ud}{\ud t}\norm{\sqrt{A}\nabla_h U}{}^2$.
The remainder of the proof follows completely analogously to the proof of Lemma \ref{lem:semi:hypo}.
\end{proof}

Hence, combining the bounds from Lemmata \ref{eq:energy_st}, \ref{eq:SUPG_st}, and \ref{eq:hypo_st}, we already deduce the following stability estimate
    \begin{equation}\label{space-time dG bound lower bound}
        B_n(U,V)
        \geq
 \frac{1}{2}\left(
	\norm{U(t_n^-)}{A,h}^2  -	\norm{U(t_{n-1}^-)}{A,h}^2 + \norm{\jumpu{U}_{n-1}}{A,h}^2  \right)
+ \frac{1}{4}\int_{I_n} \ndg{U}^2 \dt.
    \end{equation}

The last estimate motivates to define following space-time dG norm:
\begin{equation}\label{def: space-time dG norm}
    \ndg{w}_{\rm st}
  :  =\Big(
    \frac{1}{2}\Big(
	\norm{w(t_N^-)}{A,h}^2  +	\norm{w(t_0^+)}{A,h}^2 + \sum_{n=2}^N \norm{\jumpu{w}_{n-1}}{A,h}^2  \Big)
+ \frac{1}{4} \sum_{n=1}^N \int_{I_n} \ndg{w}^2 \dt\Big)^{\frac{1}{2}},
\end{equation}
for all $w|_{I_n}\in C(I_n;H^2(\Omega,\Th {))}$, $n=1,\dots,N$. Indeed, summing \eqref{space-time dG bound lower bound} over all $I_n$, $n=1,\dots, N$,  we get
    \begin{equation}\label{coerFD}
        B_{\rm st}(U,V)
        \geq
 \ndg{U}_{\rm st}^2.
    \end{equation}

Equipped with the last bound, we are ready to show the second main result.

\begin{theorem}[Asymptotic numerical hypocoercivity of space-time dG]\label{thm:discreteEquilibriation}
	Let $U \in  {\mathcal{V}}_{h,k}$ be the solution to  \eqref{eq:dG_FEM} for $f=0$. Then, we have
    \begin{equation}\label{FD_hypo}
        \norm{U(t_f)}{A,h}^2\le \prod_{n=1}^N	\Big(1+\frac{\kappa	k_n}{2 (q+1)^2}	\Big)^{-1}\norm{U(t_{0}^-)}{A,h}^2,
    \end{equation}
   with  $U(t_0^-):=\pi u_0\in  {\mathcal{V}}_h$ denoting for the approximated initial condition.
\end{theorem}
\begin{proof}
Setting $f=0$  into  \eqref{eq:dG_FEM}, for $V\in  {\mathcal{V}}_{h,k}$ given by  \eqref{FDV} on $I_n$, along with \eqref{space-time dG bound lower bound} give:
\begin{equation}\label{dG_coer}
	\frac{1}{2}\left(
	\norm{U(t_n^-)}{A,h}^2  -	\norm{U(t_{n-1}^-)}{A,h}^2  \right)
	+ \frac{1}{4}\int_{I_n} \ndg{U}^2 \dt \leq B_n(U,V) = 0.
\end{equation}
for $n=1,\dots,N$, noting carefully that for $n=1$, $U(t_0^-)=\pi u_0$ is the projected initial condition onto $ {\mathcal{V}}_h$. Hence, Lemma \ref{lemma:mesh-dependent spectral gap} implies
\begin{equation}\label{dG_stab}
\norm{U(t_n^-)}{A,h}^2+\frac{\kappa}{2}\int _{I_n}\norm{U}{A,h}^2\ud t
\le \norm{U(t_{n-1}^-)}{A,h}^2.
\end{equation}
Employing, now, the trace inverse estimate \eqref{time_trace_inv}, we get
\[
\begin{aligned}
\frac{\kappa	k_n}{2 (q+1)^2}\norm{U(t_{n}^-)}{A,h}^2
	\le &\ \frac{\kappa}{2}\int _{I_n}\norm{U}{A,h}^2\ud t,
\end{aligned}
\]
which upon substitution into \eqref{dG_stab} gives
\[
\norm{U(t_n^-)}{A,h}^2\le \Big(1+\frac{\kappa	k_n}{2(q+1)^2} \Big)^{-1}\norm{U(t_{n-1}^-)}{A,h}^2.
\]
Applying the last bound iteratively for all $n=1,\dots,N$, we deduce the required estimate.
\end{proof}

\begin{remark}
Assuming constant timestep $k_n=k$ for all $n=1,\dots,N$, \eqref{FD_hypo} gives
\begin{equation}\label{eq:convergenceToEquilibriumForFullyDiscrete}
\norm{U(t_f)}{A,h}\le 	\Big(1+\frac{\kappa	k}{2 (q+1)^2}	\Big)^{-\frac{t_f}{2k}}\norm{U(t_{0}^-)}{A,h}\to  e^{-\frac{\kappa t_f}{4(q+1)^{2}}} \norm{U(t_{0}^-)}{A,h},
\end{equation}
as $k\to 0$. Comparing with \eqref{hypoSUPG_stab_semidsicrete}, the exponent has the same scale of the respective one in the semidiscrete case for $q=0$.
 {Note that the above result also holds for variable temporal degree, say $q_n$, by replacing $q$ with $\max_n {q_n}$.}
\end{remark}

\subsection{Hypococercivity-enhanced inf-sup stability}\label{subsec: Error bound}
A by-product of the proof of Theorem \ref{thm:discreteEquilibriation} is the bound  \eqref{coerFD}, viz., $B_{\rm st}(U,V)
\geq
\ndg{U}_{\rm st}^2$, with $V$ given in \eqref{FDV}.  This is one side of the proof of an inf-sup condition in the hypocoercivity-enhanced norm $\ndg{\cdot}_{\rm st}$ for the \emph{classical} space-time dG method. To complete the proof of the inf-sup condition, we need also to show that $\ndg{V}_{\rm st}\le C\ndg{U}_{\rm st}$. This given in the final main result of this work.

\begin{theorem}\label{Theorem: Inf-sup}
There exists a positive constant $\Lambda$, independent of discretization parameters $h_T$, $k_n$, $p$ and $q$, (but depends on $\norm{x}{L_\infty(\Omega)}$,) such that:
\begin{equation}\label{eq: inf-sup}
\inf_{U\in  {\mathcal{V}}_{h,k}\backslash \{0\}} \sup_{V\in  {\mathcal{V}}_{h,k}\backslash \{0\}} \frac{B_{\rm st}(U,V)}{\ndg{U}_{\rm st}\ndg{V}_{\rm st}} \geq \Lambda.
\end{equation}
\end{theorem}
\begin{proof}
	To show this, it remains to prove that  there exists positive constant $\Lambda>0$, independent of discretization parameters $h$, $k_n$, $p$ and $q$, such that
	\begin{equation}\label{eq: upper bound}
		\ndg{V}_{\rm st}
		\leq \ndg{U}_{\rm st} + \ndg{V_2 }_{\rm st}+ \ndg{V_3}_{\rm st} \leq \Lambda^{-1}  \ndg{U}_{\rm st}.
	\end{equation}
for $V$ given in \eqref{FDV}.  This bound, together with  $B_{\rm st}(U,V)
\geq
\ndg{U}_{\rm st}^2$, proves \eqref{eq: inf-sup}.

To streamline notation, we set $\chp:= \max\{\cinvT^2,\cinv\} p^2 h_T^{-1}$, so that the inverse estimates \eqref{trace_H1_inv} can be expressed as
\begin{equation}\label{inv_est_short}
	\norm{Z}{\partial T}^2 \leq \chp  \norm{Z}{ T}^2,\quad\text{and}\quad
	\norm{\nabla Z}{T}^2 \leq  \chp^2 \norm{Z}{ T}^2,
\end{equation}
for all $Z\in\mathbb{P}_p(K)$, $p\in\mathbb{N}$, $T\in\Th $.  On each $I_n\times T$, we note the bounds
\begin{equation}\label{bounds}
	\alpha \le h_T^2\cinv^{-2}p^{-4} \leq \bar{c}\chp^{-2},\qquad  \tau_{T,n}\le \frac{2}{3}\bar{c}\chp^{-2}, \quad \text{ and }\quad
\tau_{T,n}\le
\frac{k_n}{64(q+1)^2},
\end{equation}
with $\bar{c}:=\max\{1,\cinvT^4\cinv^{-2}\}$, from the definitions of $\tau_{T,n}$ and $\delta_{T,n}$. We shall make frequent use of \eqref{inv_est_short} and \eqref{bounds} below without explicit announcement in the text.

We first show $\ndg{V_2 }_{\rm st}\leq C_2 \ndg{U}_{\rm st}$ for a constant $C_2$ independent of $h,p,q,k$. To that end,  we begin by setting, as before, $W|_{I_n\times T}:=(U_t+xU_y)|_{I_n\times T}$, $n=1,\dots,N$, $T\in\Th$, for notational brevity. Then, on each time interval  $I_n$, we have
\begin{equation*}\begin{split}
		 \norm{\sqrt{\sigma} \jump{\tau W}}{I_n\times \gint }^2
		 \leq &\
		2  \su \norm{\sqrt{\sigma} {\tau W}}{I_n\times \partial T }^2
		\\
	\leq &\
		2  \su \sigma_T \tau_{T,n} \chp\norm{ \sqrt{\tau} W}{I_n\times  T }^2
	\le
		c_1 \su \norm{ \sqrt{\tau} W}{I_n\times  T }^2 ,
\end{split}\end{equation*}
 with  $c_i>0$ (see also below), $i\in\mathbb{N}$, independent of $h$, of $p$, and of the functions involved.
Next, since $|A|_{I_n\times T}|\leq 2\alpha \le2\bar{c}\chp^{-2}$, we have
\begin{equation*}\begin{split}
	  \norm{\sqrt{A}\nabla (\tau W )_x}{I_n\times T}^2
		&\leq
	 \chp^4 |A|_{I_n\times T}| \tau_{T,n}  \norm{ \sqrt{\tau}W}{I_n\times T}^2
		\leq  c_2  \norm{\sqrt{\tau}W}{I_n\times T }^2,
\end{split}\end{equation*}
\begin{equation*}\begin{split}
	\norm{\alpha (\tau W)_y}{I_n\times T}^2
		&\leq  \chp^2 \alpha^2 \tau_{T,n}  \norm{ \sqrt{\tau} W}{I_n\times T}^2  \le c_3 \chp^{-4}\norm{\sqrt{\tau} W}{ I_n\times T }^2 .
\end{split}\end{equation*}
Similarly, we also have $\norm{(\tau W)_x}{I_n\times T}^2
 \leq
c_4\norm{ \sqrt{\tau} W}{I_n\times T}^2$.

Now, triangle inequality and the estimates \eqref{time_trace_inv} and \eqref{bounds} imply
\begin{equation*}
		   \tau_{T,n}^2 \norm{ \sqrt{\tau}(W_t + xW_y) }{I_n\times T}^2
		\leq c_5
	 \norm{ \sqrt{\tau}W  }{I_n\times T}^2,
\end{equation*}
with $c_5:= 24+
\norm{x}{L_\infty(\Omega)}^2$.
Also,
\begin{equation*}
		     \norm{ \sqrt{x\normal_2 A}\nabla (\tau W) }{I_n\times \partial_+ T}^2 +	 \norm{ \sqrt{|x\normal_2|} \tau W}{I_n\times \partial T }^2
		\leq
	c_6  \norm{ \sqrt{\tau}W }{I_n\times  T}^2.
	  \end{equation*}

Finally, for the time discretisation terms, using \eqref{bounds}, we have
\begin{equation*}\begin{split}
			\norm{\tau W(t_{n-1}^\pm)|_{T}}{A}^2 = &\
		\norm{\tau W(t_{n-1}^\pm)}{T}^2
		+
		\norm{\tau \sqrt{A}\nabla W(t_{n-1}^\pm)}{T}^2
		\\
			\leq&\  (q+1)^2 \tau_{T,m}k_m^{-1}
		(1
		+ 2\bar{c} ) \norm{\sqrt{\tau}W }{I_{m}\times T}^2 \le c_7 \norm{\sqrt{\tau}W }{I_{m}\times T}^2,
\end{split}\end{equation*}
with $m:=n+\frac{\pm 1-1}{2}$, i.e., $m=n$ for $t_{n-1}^+$ and $m=n-1$ for $t_{n-1}^-$. Combining the above bounds, we arrive at $\ndg{V_2 }_{\rm st}\leq C_2 \ndg{U}_{\rm st}$ for $C_2>0$ independent of $h,p,q,k$.

 To conclude, we also need to show that $\ndg{V_3 }_{\rm st}\leq C_3 \ndg{U}_{\rm st}$ for a constant $C_3>0$ independent of $h,p,q,k$. We start by observing that
\begin{equation}\label{eq: A norm bound}
	\begin{split}
		\|\sqrt{A}\nabla U\|_T^2
		& =  \alpha \|U_x\|_T^2 + \frac{8\alpha^2}{3} (U_x, U_y)_T + \frac{20\alpha^3}{9}   \|U_y\|_T^2
		\leq 2\alpha \|U_x\|_T^2  +  \frac{28\alpha^3}{9}   \|U_y\|_T^2,
	\end{split}
\end{equation}
or 	$\|\sqrt{A}\nabla U\|_T^2 \le 4\alpha  \norm{\sqrt{B}\nabla U}{T}^2$, with $B= {\rm diag}(1,\alpha^2)$ as in the proof of Lemma \ref{lemma:mesh-dependent spectral gap}.
We consider each term in $\ndg{V_3 }_{\rm st}^2$ separately, starting from
\begin{equation*}
		\norm{\sqrt{\sigma} \jump{\Div(A \nabla U)}}{I_n\times \gint }^2
		\leq
\su
		4\sigma_T \chp^3 |A|_{I_n\times T}|   \norm{\sqrt{A} \nabla U}{ I_n\times T }^2
		 \le  \tilde{c}_1\norm{\sqrt{B}\nabla_h U}{}^2 ,
\end{equation*}
with $\tilde{c}_i>0$ (see also below), $i\in\mathbb{N}$, independent of $h,p,q,k$ and of the functions involved.
Also, we have, respectively,
\begin{equation*}
		   \norm{\sqrt{A}\nabla \left(\Div(A\nabla U) \right)_x}{I_n\times T}^2
		\leq
	4	\alpha \chp^6  |A|_{I_n\times T}|^2  \norm{\sqrt{B}\nabla U}{I_n\times T}^2\le \tilde{c}_2 \norm{\sqrt{B}\nabla U}{I_n\times T}^2,
\end{equation*}
\begin{equation*}
		  \norm{\alpha  \big(\Div(A\nabla U) \big)_y}{I_n\times T}^2
		\leq  4  \alpha^3\chp^4 |A|_{I_n\times T}|  \norm{\sqrt{B}\nabla U}{I_n\times T}^2
			\leq  \tilde{c}_3  \chp^{-4} \norm{\sqrt{B}\nabla U}{I_n\times T}^2,
\end{equation*}
and
\begin{equation*}
		\norm{\Div(A\nabla U)_x}{I_n\times T}^2
	     \leq
 4 \alpha\chp^4 |A|_{I_n\times T}|  \norm{\sqrt{B}\nabla U}{I_n\times T}^2\le \tilde{c}_4 \norm{\sqrt{B}\nabla U}{I_n\times T}^2.
\end{equation*}
Further, recalling \eqref{final_bound_alpha}, we see that $\alpha< \underline{c}\sqrt{k_n}(q+1)^{-1}\chp^{-1}$,  with $\underline{c}:= \max\{1,\cinv\cinvT^{-2}\}$, so that
\begin{equation*}\begin{split}
		&   \norm{ \sqrt{\tau}((\Div(A\nabla U))_t + x(\Div(A\nabla U))_y) }{I_n\times T}^2\\
		\leq&  8\alpha \chp^2
	\big(   12(q+1)^4k_{n}^{-2}    \! +\!
		\norm{{x}}{L_\infty(T)}^2 \chp^4\big)|A|_{I_n\times T}|\tau_{T,n}\norm{\sqrt{B}\nabla U}{I_n\times T}^2
	 \leq   \tilde{c}_{5}   \norm{\sqrt{B}\nabla U}{I_n\times T}^2.
\end{split}\end{equation*}
Also,  we have
\begin{equation*}\begin{split}
		& \norm{ \sqrt{x\normal_2 A}\nabla (\Div(A\nabla U)) }{I_n\times \partial_+ T}^2 + \norm{ \sqrt{|x\normal_2|} \Div(A\nabla U)}{I_n\times \partial T}^2 \\
		\leq &\
\norm{x}{L_\infty(\Omega)}^2 \big(	 \chp^5 |A|_{I_n\times T}|^2
+ \chp^3 |A|_{I_n\times T}|  \big) \norm{ \sqrt{A}\nabla U }{I_n\times T}^2
	  \leq   \tilde{c}_{6} \chp^{-1}  \norm{\sqrt{B}\nabla U}{I_n\times T}^2.
\end{split}\end{equation*}
Finally,  employing again \eqref{bounds} and $\alpha< \underline{c}\sqrt{k_n}(q+1)^{-1}\chp^{-1}$, we have
\begin{equation*}\begin{split}
			& \norm{\Div(A\nabla U)(t_{n-1}^\pm)|_{T}}{A}^2 \\
			=&\
		\norm{\Div(A\nabla U)(t_{n-1}^\pm)}{T}^2
		+\norm{\sqrt{A}\nabla\Div(A\nabla U)(t_{n-1}^\pm)}{ T}^2
		\\
		\leq &\
  \frac{(q+1)^2}{k_m}  \Big(  \chp^2 |A|_{I_m\times T}|
		+ \chp^4 |A|_{I_m\times T}|^2   \Big) \norm{\sqrt{A} \nabla U }{I_m\times  T }^2
		  \leq   \tilde{c}_{7} \norm{\sqrt{B}\nabla U}{I_m\times T}^2,
\end{split}\end{equation*}
for $m:=n+\frac{\pm 1-1}{2}$.   Combining the above bounds, we arrive at $\ndg{V_3}_{\rm st}\leq C_3 \ndg{U}_{\rm st}$ for $C_2>0$ independent of $h,p,q,k$, which concludes the proof.
\end{proof}

\begin{remark}
	We note that Theorems \ref{thm:discreteEquilibriation} and \ref{Theorem: Inf-sup} remain valid (upon setting $p=q$), with different constants in the definitions of $\tau$ and $\delta$, if we consider \emph{total} degree space-time polynomials in the spirit of \cite{MR3672375}.
\end{remark}

\subsection{\emph{A priori} error bound}

Set $e := u-U$. The consistency of \eqref{eq:dG_FEM} implies:
\begin{equation}\label{eq:ErrorEquation}
\sum_{n=1}^N \int_{I_n}\!\! \big( \inner{e_t}{V}{} + a_{\rm dG}(e,V) \big) \dt + \sum_{n=2}^{N} \inner{\tjump{e}_{n-1}}{V( t_{n-1}^+)}{}
+	\inner{e(t_0^+)}{V(t_0^+)}{}
=
0,
\end{equation}
for all $V\in \mathcal{V}_{h,k}$.

	We define the space-time  projection  $\pi_{\rm st}: = \pi^t_q\otimes \pi^s_p$, whereby $\pi^t_q|_{I_n} \colon L_2(I_n) \to \mathbb{P}_q(I_n)$ is the $L_2$-orthogonal projection on each $I_n$, $n=1,\dots,N$, and  $\pi^s_p|_{T} \colon L_2(T) \to \mathbb{P}_p(T)$ is the local $L_2$-orthogonal projection on each $T\in \Th$. In particular, for all $v\in L^2(I_n; L^2(\Omega))$ and for all $W \in \mathcal{V}_{h,k}$, we have
\begin{equation}\label{def: space-time L2 orthogonal projection}
	\int_{I_n} (v,W)_T \ud t
				=
				\int_{I_n} (\pi_{\rm st}{v},{W})_T \ud t.
\end{equation}

The approximation properties of $\pi_{\rm st}$ can be immediately proven by a tensor-product argument and known approximation properties of the $L_2$-projection error, along with the stability of $ \pi^t_q$. Setting  \[
\mathcal{W}_{r,s,\epsilon}^n:=H^1(I_n;H^{s+1}(\Omega,\Th)) \cap H^{r}(I_n;H^{2.5+\epsilon}(\Omega)),
\] the following best approximation estimates hold:
\begin{equation}\label{eq:space_time_approximation}
	\begin{aligned}
&\int_{I_n}
%\sum_{T\in\mathcal{T}}
\left( \norm{D^{\lambda} (v-\pi_{\rm st}v)}{T}^2
%\sum_{T\in\mathcal{T}}
+
\norm{\sqrt{h}D^\lambda (v-\pi_{\rm st}v)}{\partial T}^2
\right)\ud t \\
\le&\  C_{\rm app} \int_{I_n}
\big( k^{2r}_n\norm{ \partial_t^r D^\lambda v}{T}^2
+ h^{2(s+1-\lambda)}_T|v|_{H^{s+1}(T)}^2
\big)\ud t,
	\end{aligned}
\end{equation}
for any $v\in\mathcal{W}_{r,s,\epsilon}^n$,
%and using the multiplicative trace inequality, we deduce
%\begin{equation}\label{eq:space_time_approximation trace}
%	\begin{aligned}
%\int_{I_n}
%%\sum_{T\in\mathcal{T}}
%\norm{h^{\frac{1}{2}}\nabla^\lambda (v-\pi_{\rm st}v)}{\partial T}^2
%\ud t
%\le&  C_{app} \int_{I_n}
%\big(k_n^{2r} |   \partial_t^r  v|_{H^{\lambda}(\partial T)}^2
%+ h_T^{2(s+1-\lambda)}|v|_{H^{s+1}( T)}^2
%\big)\ud t ,
%	\end{aligned}
%\end{equation}
$\epsilon>0$, $n=1,\dots,N$, $1\le r\le q+1$, $1.5< s\le p$, $p\geq2$ and $\lambda\leq 2$, and $C_{\rm app}>0$ independent of $h$, of $k_n$ and of $v$, with $\partial_t^r$ denoting the $r$-th derivative with respect to $t$ and $D^\lambda$ is a spatial derivative with respect to a multi-index $\lambda$. For simplicity of the presentation, we absorb the (known) dependence on $p$ and on $q$ into the unspecified constants; $hp$-version bounds are by all means possible by careful tracking of $p$ and of $q$.

Next, %using \eqref{eq: A norm bound}, along with
the bound $|A|_{T}|\leq 2\alpha \le C {h_T^2/p^4}$, \eqref{inv_est_short}, and the stability of $L^2$-orthogonal projection $ \pi^s_p$, give, for any $v\in H^1(I_n;H^{1}(\Omega,\Th)) $,
\[
\norm{\sqrt{A} \nabla (\pi^s_p v)}{T}^2
\le   |A|_{T}| \norm{\nabla (\pi^s_p v)}{T}^2
\le C  \norm{\pi^s_p v}{T}^2
\le C  \norm{v}{T}^2,
\]
and
\[
\norm{\sqrt{A} \nabla (\pi^s_p v)_x}{T}^2
\le   |A|_{T}| \norm{(\nabla \pi^s_p v)_x}{T}^2
\le  \norm{\nabla \pi^s_p v}{T}^2
\le C  \norm{\nabla v}{T}^2.
\]
Using the above stability, we have for $t\in\bar{I}_n$:
\begin{equation}\label{eq:space_time_approximation new}
	\begin{aligned}
		\norm{(v-\pi_{\rm st}v)(t)}{A,h}^2
		\leq &\
		C_{\rm app} \Big(\su
		h_T^{2(s+1)}|v(t)|_{H^{s+1}(T)}^2
		+
		\int_{I_n}
		k_n^{2r-1} \|   \partial_t^r  v \|_{A,h}^2
		\ud t\Big),
	\end{aligned}
\end{equation}
and
\begin{equation}\label{eq:space_time_approximation new_new}
	\begin{aligned}
	& \int_{I_n}\Big( \norm{\sqrt{A}\nabla(v-\pi_{\rm st}v)_x}{T}^2+ \norm{ \sqrt{A}\nabla (v-\pi_{\rm st}v) }{\partial T}^2\Big)\ud t\\
		\leq &\
		C_{\rm app} \int_{I_n} \Big(
		h_T^{2s}|v|_{H^{s+1}(T)}^2
		+
		k_n^{2r}h_T(1+h_T)\|   \partial_t^r  \nabla^2 v \|_{T}^2
		\Big)\ud t,
	\end{aligned}
\end{equation}
for any $v\in\mathcal{W}_{r,s,\epsilon}^n$, $\epsilon>0$, $n=1,\dots,N$, $1\le r\le q+1$, $1.5< s\le p$, $p\geq2$, $\lambda\leq 2$.
\begin{theorem}
Let $\max_{T}h_T<1$ and $\max_{n}k_n<1$. If the exact solution satisfies $u|_{I_n}\in \mathcal{W}_{r,s,\epsilon}^n$, with $\epsilon>0$, $1\le r\le q+1$, $1.5< s\le p$, $p\geq2$, $n=1,\dots, N$, we have the bound
\begin{equation}\label{eqn: error bound in ST norm}
\begin{split}
	\ndg{u-U}_{\rm st} \leq C \bigg(\sum_{n = 1}^N \mathbb{E}_n\bigg)^{\frac12},
\end{split}
\end{equation}
where
\begin{equation*}
	\begin{aligned}
		\mathbb{E}_n := &\   \su \int_{I_n}\bigg(h_T^{2s}\Big( k_nh_T^2|\partial_t u|_{H^{s+1}(T)}^2+(1+k_n^{-1}h_T^2)|u|_{H^{s+1}(T)}^2\Big)
		% +\int_{I_n}\!k_n^{2r-1} \|   \partial_t^r  v \|_{A,h}^2\ud t\bigg)
		\\
		&\qquad+
		k_n^{2r-1}\Big( (k_n+h_T^2) \norm{\partial_t^r \nabla v}{T}^2+k_nh_T\|   \partial_t^r  \nabla^2 v \|_{T}^2
+ k_n h_T^{-1}\norm{ \partial_t^r  v}{T}^2\Big)
		\bigg)\ud t
	\end{aligned}
\end{equation*}
with $C>0$, independent of $h_T$, $k_n$, and of $u$, but dependent on $p$, $q$, $\Omega$ and on $C_{\rho}$.
\end{theorem}
\begin{proof}
	We begin by decomposing the error as follows: $e = \eta + \xi$, with $\eta= u - \pi_{\rm st}u$ and $\xi= \pi_{\rm st}u - U$. The triangle inequality, Galerkin orthogonality, and the inf-sup stability \eqref{eq: inf-sup}, imply
	\begin{equation}\label{eq: abstract error bound}
		\begin{split}
			\ndg{u-U}_{\rm st} \leq \ndg{\xi}_{\rm st} + \ndg{\eta}_{\rm st}
			\leq \frac{1}{\Lambda}  \sup_{V\in  {\mathcal{V}}_{h,k}\backslash \{0\}} \frac{|B_{\rm st}(\eta,V)|}{\ndg{V}_{\rm st}}
			+\ndg{\eta}_{\rm st}.
		\end{split}
	\end{equation}

First, we estimate the first term on the right-hand side of \eqref{eq: abstract error bound} from above.
Integration by parts for the time derivative in \eqref{eq:space-time bilinear form}, gives
\begin{equation*}
\begin{split}
&   \sum_{n=1}^N \int_{I_n} \inner{\eta_t}{V }{} \dt + \sum_{n=2}^{N} \inner{\tjump{\eta}_{n-1}}{V ( t_{n-1}^+)}{}
+	\inner{\eta(t_0^+)}{V (t_0^+)}{} \\
=&  -\sum_{n=1}^N \int_{I_n} \inner{\eta}{V _t}{} \dt - \sum_{n=2}^{N} \inner{\eta( t_{n-1}^-)}{\tjump{V }_{n-1}}{}
+	\inner{\eta(t_N^-)}{V (t_N^-)}{},
\end{split}
\end{equation*}
where, since $ V_t \in  {\mathcal{V}}_{h,k}$,  we have  $\int_{I_n}
\inner{\eta}{V_t}{T}
\ud t = 0$.
Furthermore, for $V  \in  {\mathcal{V}}_{h,k}$, \eqref{def: space-time L2 orthogonal projection}, implies
\[
\int_{I_n} a_{\rm dG}(\eta,V ) \ud t =  \int_{I_n} a_{\rm dG}(\pi_q^t\eta,V ) \ud t,
\]
where we note that $\pi_q^t\eta = \pi_q^t(u-\pi_p^s u)$.

Element-wise integration by parts for the advection term gives
\begin{equation*}
	\begin{aligned}
		&	\sum_{T \in \Th} \big( \inner{x \pi_q^t\eta_y}{V}{T}-
		\inner{x\normal_2 \pi_q^t\jumpu{\eta}}{V^+}{\partial_- T \cap \gint}
		+
		\inner{x\normal_2 \pi_q^t\eta^+ }{ V^+ }{\partial_- T \cap \gif} \big)\\
		=& \  	\sum_{T \in \Th} \big(- \inner{\pi_q^t\eta}{x V_y}{T}  +  \inner{x\normal_2  \pi_q^t\eta^-}{ \jumpu{V}}{\partial_- T \cap \gint}
		+
		\inner{x\normal_2 \pi_q^t\eta^+ }{ V^+ }{\partial_+ T \cap \Gamma}
		\big).
	\end{aligned}
\end{equation*}
Also, since $V  \in  {\mathcal{V}}_{h,k}$ and, thus, $x V _y \in  {\mathcal{V}}_{h,k}$, the orthogonality \eqref{def: space-time L2 orthogonal projection} implies
\begin{equation*}
	\int_{I_n}
	\inner{\pi_q^t\eta}{x V _y}{T}
	\ud t = 0.
\end{equation*}

Thus, using the stability of $\pi_q^t$, and performing standard estimation, we deduce
\begin{equation*}\label{bound one on a bilinear form}
	\begin{aligned}
		| \int_{I_n} a_{\rm dG}(\eta,V ) \ud t|
		\leq & \
		C\Big(
		\int_{I_n}\su\Big(
		\norm{\pi_q^t\eta_x}{T}^2 + \norm{\sigma^{-\frac12}\pi_q^t\eta_x}{\partial T\cap \gint}^2 \\
		&+ \norm{\sqrt{\sigma}\pi_q^t\eta}{\partial T\cap \gint}^2		+ \norm{\sqrt{|x\normal_2|} \pi_q^t\eta}{\partial T}^2\Big)
		\ud t
		\Big)^{\frac12}
		\left(
		\int_{I_n}
		\ndg{V }^2
		\ud t
		\right)^{\frac12}\\
		\leq & \
		C\Big(
		\int_{I_n}\su\mathcal{E}_{n,T}^2
		\ud t
		\Big)^{\frac12}
		\left(
		\int_{I_n}
		\ndg{V }^2
		\ud t
		\right)^{\frac12},
	\end{aligned}
\end{equation*}
with
\[
\begin{aligned}
\mathcal{E}_{n,T}:=&\
\big(\norm{(u-\pi_p^su)_x}{T}^2 + \norm{\sigma^{-\frac12}(u-\pi_p^su)_x}{\partial T\cap \gint}^2\\
&\ + \norm{\sqrt{\sigma}(u-\pi_p^su)}{\partial T\cap \gint}^2
+ \norm{\sqrt{|x\normal_2|} (u-\pi_p^su)}{\partial T}^2\big)^{\frac12}.
\end{aligned}
\]
The above, along with application of \eqref{bound one on a bilinear form},  result in the bound
\begin{equation}\label{def: error bound for xi}
	\begin{aligned}
 \frac{1}{\Lambda}\sup_{V\in  {\mathcal{V}}_{h,k}\backslash \{0\}}\!\frac{|B_{\rm st}(\eta,V )|}{\ndg{V }_{\rm st}}
\leq  C\Big(
\sum_{n=1}^N
\su\!\Big( \int_{I_n}\mathcal{E}_{n,T}^2\ud t  +\norm{\eta(t_{n}^-)}{T}^2\Big)
\Big)^{\frac12}
.
\end{aligned}
\end{equation}

Employing standard best approximation estimates, the right-hand side of \eqref{def: error bound for xi} can be bounded above by
\begin{equation}\label{eq: xi bound Error time step}
	\begin{aligned}
 &    C\Big(
\sum_{n=1}^N \su\Big(
 \int_{I_n}
 h_T^{2s}| u|^2_{H^{s+1}(T)}
\ud t+ h_T^{2(s+1)}
| u(t_{n}^-)|^2_{H^{s+1}(T)}\Big)
\Big)^{\frac12}
,
	\end{aligned}
\end{equation}
with $C>0$ depending on $p$ and on $C_\rho$.

Now, utilising the trace estimate with respect to the time variable, we have
\[
\begin{aligned}
	   \ndg{\eta}_{\rm st}^2
\le&\  \sum_{n=1}^N\Big(\norm{\eta(t_n^-)}{A,h}^2+\norm{\eta(t_{n-1}^+)}{A,h}^2+ \int_{I_n} \ndg{\eta}^2 \dt\Big)
\end{aligned}
\]
The continuity of $u$ and of $\pi_q^t u$ with respect to the spatial variables implies $ \jump{\eta}= \jump{\pi_q^t  u- \pi_{\rm st}u}=\jump{\pi_q^t  (u- \pi_p^su)}$ on $\gint$. The latter along with a trace estimate and standard arguments, gives
\[
\begin{split}
	\ndg{\eta}^2
	\le&\  C\!\!
	 \sum_{T \in \Th} \Big(
	 \norm{\sqrt{A}\nabla \eta_x}{T}^2
	 +
	 \norm{ \sqrt{A}\nabla \eta }{\partial T}^2
	 +
h_T^2 \norm{\eta_y}{T}^2
	 +
	 \norm{\eta_x}{T}^2 + k_n\norm{ \eta_t }{T}^2
+\norm{\eta}{\partial T}^2  \Big)\\
	 &
	 + \norm{\sqrt{\sigma+|xn_2|} \jump{\pi_q^t(u- \pi_p^su)}}{\gint }^2,
\end{split}
\]
for constant $C>0$ depending on $p$, on $q$ and on $\Omega$, upon
recalling the definition of $\alpha$ in \eqref{eq:alpha_beta_gamma_def}, of $\delta_{T,n}$ in \eqref{def:deltaFD_def}, and of $\tau_{T,n}$ in \eqref{def: tau_st}. Further application of a trace estimate and the stability of $\pi_q^t$ gives
\begin{equation}\label{sigma_term}
 \int_{I_n}\norm{\sqrt{\sigma+|xn_2|} \jump{\pi_q^t(u- \pi_p^su)}}{\gint }^2\ud t\le 
 C \su  \int_{I_n}h_T^{2s} |u|_{H^{s+1}(T)}^2\ud t,
\end{equation}
with $C>0$, depending only on $p$, $s$, $\Omega$, and $C_\rho$. Also,
\[
\int_{I_n} k_n\norm{ \eta_t }{T}^2 	\ud t \le C\int_{I_n}   \Big( k_n^{2r-1}\norm{\partial_t^r u }{T}^2
+k_nh_T^{2(s+1)}|\partial_t u|_{H^{s+1}(T)}^2	\Big)
\ud t
\]
while, from \eqref{eq:space_time_approximation}, we have
\[
\int_{I_n} \!\norm{\eta}{\partial T}^2	\ud t
\le
C
	\int_{I_n}\!
\big( k^{2r}_nh_T^{-1}\norm{ \partial_t^r  u}{T}^2
+ h^{2s+1}_T|u|_{H^{s+1}(T)}^2
\big)\ud t.
\]
The result follows by combining the above bounds, along with \eqref{eq:space_time_approximation new} and \eqref{eq:space_time_approximation new_new},
and the trace estimate
\[
|u(t_m^\pm)|_{H^{s+1}(T)}^2\le C\int_{I_n}\big( k_n^{-1}|u|_{H^{s+1}(T)}^2+ k_n |\partial_tu|_{H^{s+1}(T)}^2\big)\ud t,
\] for $m=1,\dots,N$.
\end{proof}

\begin{remark}
	The use of the $L_2$-projection in time is crucial for the above proof; it  was first used in \cite{chrysafinos_hou} to obtain symmetric  estimates for dG-timestepping.
\end{remark}

\begin{remark}\label{rem:optimalityOfhk}
The fully-discrete error bound \eqref{eqn: error bound in ST norm} is optimal with respect to the spatial mesh size $h$ in the space--time dG norm, as well as with respect to the time-step size $k$ under the condition $h \approx \mathcal{O}(k)$. Indeed, for sufficiently regular solutions, and assuming quasi-uniform spatial and temporal meshes such that $h \approx \mathcal{O}(k)$, the last theorem implies the optimal \emph{a priori} error bound
\begin{equation}\label{error bound Space-Time DG norm}
	\ndg{u-U}_{\rm st}
\leq C \big(h^{p} + k^{q+1/2} \big).
\end{equation}
\end{remark}

\section{Numerical experiments}\label{sec: numerical experiments}

We present some numerical experiments, aiming to asses the convergence rates predicted by the theory.
We measure the error in the space-time dG norm  and we demonstrate the exponential decay of the broken $A$-norm of a problem with no forcing to investigate the validity of \eqref{eq:convergenceToEquilibriumForFullyDiscrete}.

\subsection{Example 1: fully-discrete error}
Let $\Omega = (-1,1)^2$, partitioned by a sequence of successively refined uniform meshes consisting of  16, 64, 256, 1024, and 4096 quadrilateral elements and let $k =h/ \sqrt{2}$; cf., Remark \ref{rem:optimalityOfhk}. We set $t_f =1$ and choose $f$, so that
\begin{equation*}
\begin{split}
u(t,x,y) =& \exp\left(-(x-0.5)^2 - (y-0.5)^2 \right) \sin^2(\pi x) \sin^2(\pi y) + 2.
\end{split}
\end{equation*}
The convergence history of  $\ndg{e}_{\rm st}$ against the meshzise $h$ is given in Figure \ref{fig:instationaryExperiment}. Optimal convergence rate $\mathcal{O}(h^p)$ for $p=1,2,3,4$ is also observed  upon  selecting polynomial order $q = 0,1,2,3$.
\begin{figure}\centering
\includegraphics[width = .5\linewidth]{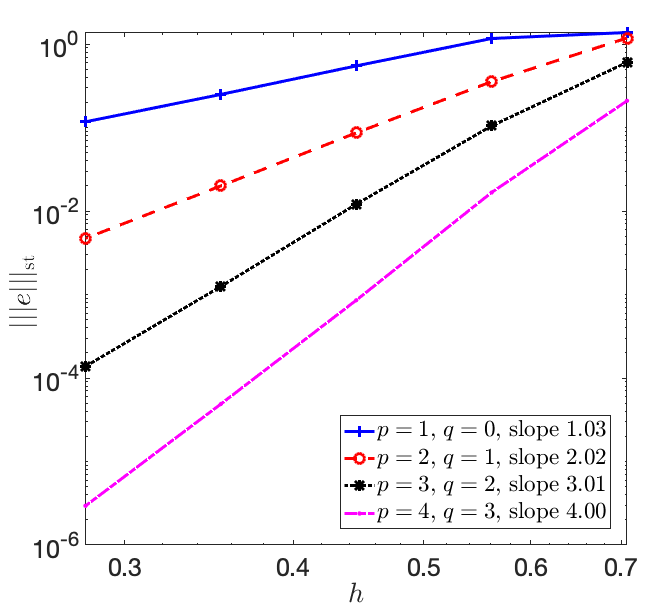}
\caption{Example 1. Convergence history.}\label{fig:instationaryExperiment}
\end{figure}

\subsection{Example 2: exponential convergence to equilibrium}
We now set  $\Omega = (0,1)^2$, $f \equiv 0$ and $t_f = 5$.
Let also $k = h/ \sqrt{2}$.
We consider the low-regularity initial condition
\begin{equation}\label{ex2:step function}
\begin{split}
u_0(x,y) =& 10\sin(\pi x)^2 \sin(\pi y)^2 \Big(\exp\big(-100(x-0.5)^2-100(y-0.5)^2 \big)\\
&+ 2\exp \big(-100(x-0.25)^2-100(y-0.25)^2 \big) + 10\times \chi_{[0.25,0.75]^2}
\Big),
\end{split}
\end{equation}
with $\chi_{[0.25,0.75]^2}$ denoting an indicator function with value one on $[0.25,0.75]^2$ and zero otherwise.
\begin{figure}\centering
\includegraphics[width = .48\linewidth]{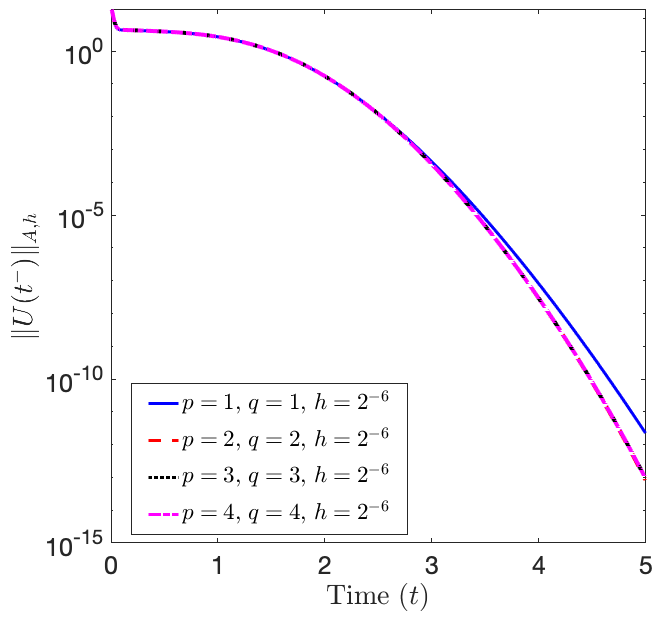}
\includegraphics[width = .48\linewidth]{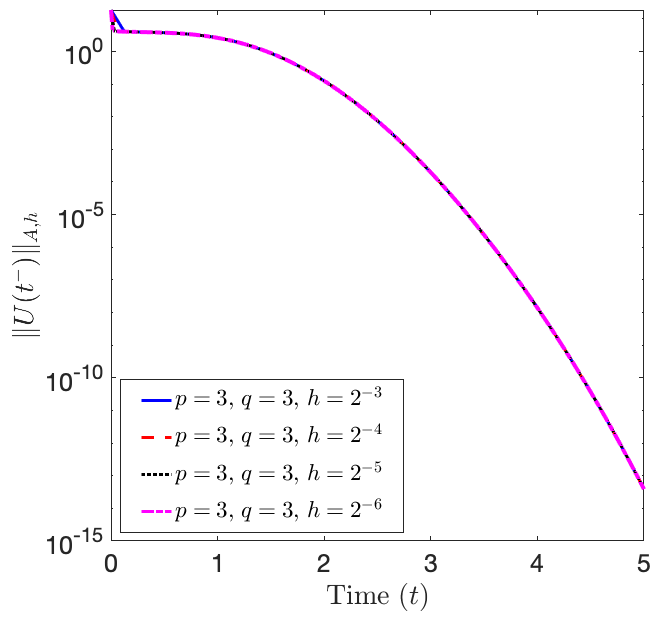}
\caption{Example 2. Convergence to equilibrium for different values of the discretization parameters in semilogy scale with $t_f=5$.}\label{fig:convergenceToEquilibrium}
\end{figure}
\begin{figure}\centering
\includegraphics[width = .48\linewidth]{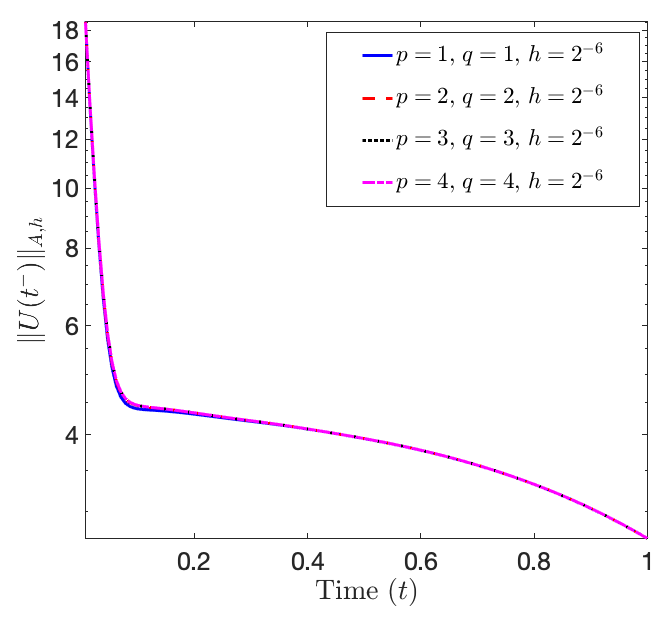}
\caption{Example 2. Convergence to equilibrium for different values of the discretization parameters with $t_f=1$.}\label{fig:convergenceToEquilibrium_short_time}
\end{figure}
In Figure \ref{fig:convergenceToEquilibrium} we provide the values $\norm{U(t_n^-)}{A}$ against $t_n$, for $n=0,\dots, N$.
The exponential convergence observed in this particular example is faster than predicted by \eqref{eq:convergenceToEquilibriumForFullyDiscrete} for $t_f=5$ asymptotically, as it does not appear to be adversely dependent on $h$ and/or $q$.

In Figure \ref{fig:convergenceToEquilibrium_short_time} is a zoom in of Figure \ref{fig:convergenceToEquilibrium} (left panel) for $t\in[0,1]$. It remains an open question whether the degeneration of exponential decay for small $h$ and/or large $p$ can be witnessed in carefully constructed numerical examples and it is a direction of future research.

Finally, we present the numerical solutions computed on 4096 rectangular elements with $p=q=4$ and $k = h/\sqrt{2}$ at the time nodes $t=0$, $t=0.3359$, $t=0.6719$, and $t=1$. In Figure~\ref{fig:numerical lolution at different time steps}, we observe that the numerical solution at $t=0$ exhibits a discontinuity due to the choice of the initial condition $u_0$ in~\eqref{ex2:step function}. At later times, namely $t=0.3359$, $t=0.6719$, and $t=1$, the numerical solution becomes smooth, while its magnitude decreases from approximately $15$ to $6$, which is consistent with the decay observed in Figure~\ref{fig:convergenceToEquilibrium_short_time}.
\begin{figure}\centering
\begin{tabular}{cc}
\includegraphics[width = .48\linewidth]{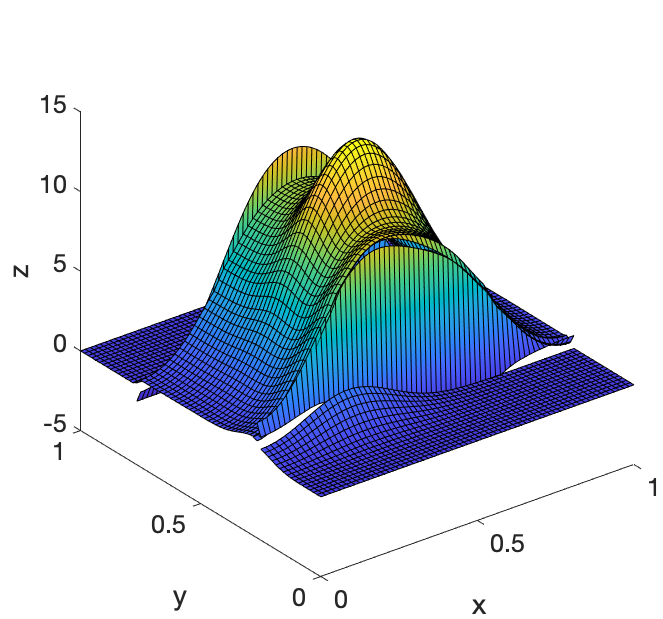}&
\includegraphics[width = .48\linewidth]{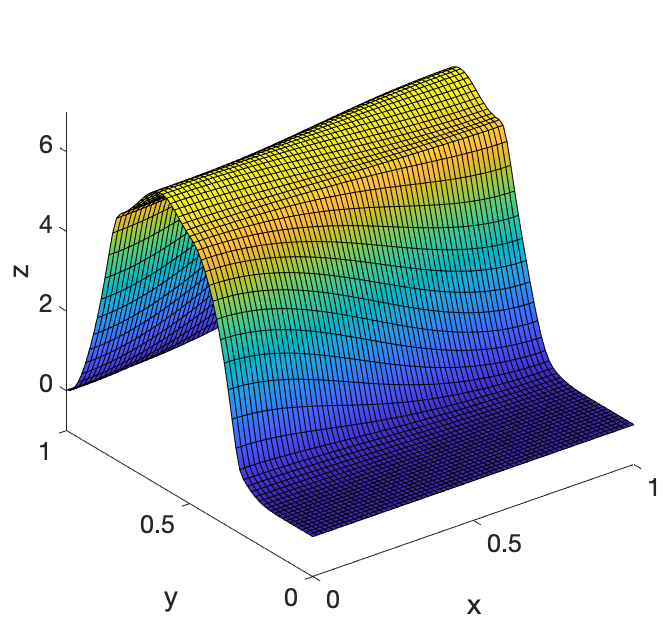}\\
$t=0$ &
$t=0.3359$\\
\includegraphics[width = .48\linewidth]{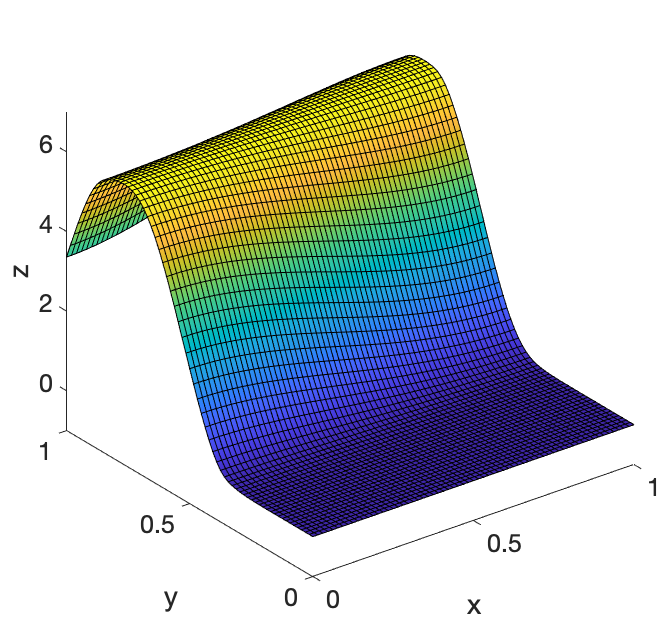}&
\includegraphics[width = .48\linewidth]{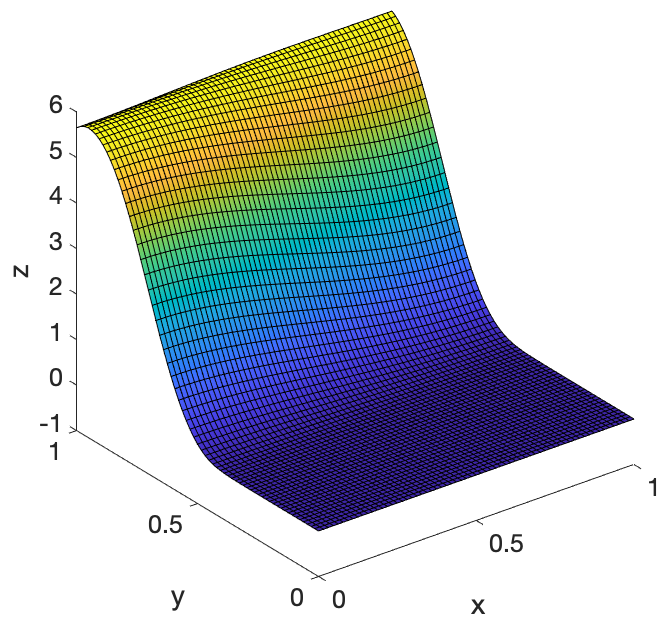}\\
$t=0.6718$ &
$t=1$\\
\end{tabular}
\caption{Example 2. Numerical solution for $4096$ rectangular elements with $p=q=4$.}\label{fig:numerical lolution at different time steps}
\end{figure}

\section{Conclusions and extensions} \label{sec: conclusion and extension}
We have shown that the standard, classical, space-time discontinuous Galerkin method discretising the Kolmogorov equation on a bounded spatial domain admits a subtle, mesh-dependent hypocoercivity property at the discrete level, asymptotically for large times. To our knowledge, this is the first result of this kind for any standard Galerkin scheme. Kolmogorov initial/boundary value problem on a bounded domain with standard no-flux conditions is not expected to admit decaying solutions for \emph{all} initial conditions. The asymptotic numerical hypocoercivity results proven above are consistent with the latter statement, in the sense that the decay as $t\to\infty$ proven degenerates as $h/p,k/q\to 0$. Nevertheless, the stability proven in Theorems \ref{main_SD}  and \ref{thm:discreteEquilibriation} is strictly stronger than the corresponding strongest inf-sup stability result from the literature \cite{johnson_pitkaranta, ayuso_marini,cdgh16} described in Remark \ref{remark_stab}. The key development that allows for the stronger results is to reveal the underlying hypocoercivity structure through the new test functions \eqref{semi: test function} and \eqref{FDV}. Further, we have shown that the space-time discontinuous Galerkin method is inf-sup stable in the same hypocoercivity-enhanced discretisation parameter-dependent family of norms containing a full gradient of the numerical solution. Also, \emph{a priori} error bounds for sufficiently regular exact solutions $u$ are presented in the `enhanced' space-time norm \eqref{def: space-time dG norm}, using the new inf-sup condition.

It is possible, in principle, to extend the results presented to meshes comprising general polygonal elements, with the respective choice of discontinuity penalisation parameter $\sigma$ as in \cite{cdgh16,cdgh_book}, provided a  broken Poincar\'e inequality on such meshes is available; we refer to a recent result in this direction \cite{botti2025sobolev}. We have refrained from doing so in this work, to avoid additional technicalities, in the interest of clarity of the presentation.

Finally, a note regarding the generality of the approach. Although the focus of this work has been on the specific Kolmogorov equation, the methodology used, i.e., to enhance the test functions by a `hypocoercivity-inducing' term mimicking the effect of the modified inner product in \eqref{modified_weak_form}, is general, in principle. As such, we expect to be able to prove similar results for other hypocoercivity-inducing equations, kinetic-type advection fields.

\appendix
\section{An elementary eigenvalue bound}
 \begin{lemma}\label{eigen_est}
 	Consider a real $2\times 2$ symmetric and positive-semidefinite matrix $\scriptsize\mathcal{A}=\begin{pmatrix}
 		a & b \\
 		b & c
 	\end{pmatrix},
 $ whose entries satisfy
$a,c\ge0$. Then, we have $
\frac{ac-b^2}{a+|b|}\le |\mathcal{A}| \le a+|b|\le 2a$, with $|\cdot|$ the spectral matrix norm.
 \end{lemma}
 \begin{proof}
 	The assumptions imply $a\ge 0$ and $ac-b^2\ge 0$. For $\lambda_-\le \lambda_{+}$ the two eigenvalues of $\mathcal{A}$, we have
 	$
 	2\lambda_{\pm}
 	=(a+c )(1 \pm \sqrt{1-4\zeta})
 	$,
 	with $\zeta:=\frac{ac-b^2}{(a+c)^2}$, we get $\lambda_{+} \le a+|b|$.  Since $\lambda_{+}\lambda_-=\det (\mathcal{A})=ac-b^2$, we conclude
 	$
 	\lambda_-\ge \frac{ac-b^2}{a+|b|}.
 	$
 \end{proof}
\bibliographystyle{amsplain}
\bibliography{ref}

\end{document}